\documentclass[a4paper,reqno]{amsart}

\title{Elliptic Boundary Value Problems and \\Partial Group Actions}
\author[E. Ewert]{Eske Ewert}
\address{Eske Ewert. 
Institute of Analysis, Leibniz University Hannover,  Welfengarten 1, 30167 Hannover, Germany}
\email{ewert@math.uni-hannover.de}

\author[A.Yu. Savin]{Anton Yu. Savin}
 \address{Anton Yu. Savin.  
Peoples' Friendship University of Russia (RUDN University),
6 Miklukho-Maklaya St
117198 Moscow, Russia
}
\email{antonsavin@mail.ru}

\author[E. Schrohe]{Elmar Schrohe}
 \address{Elmar Schrohe. 
Institute of Analysis, Leibniz University Hannover,  Welfengarten 1, 30167 Hannover, Germany}
\email{schrohe@math.uni-hannover.de}

\usepackage[utf8]{inputenc}
\usepackage[english]{babel}

\usepackage{amsmath}
\usepackage{amssymb}
\usepackage[alphabetic]{amsrefs}
\usepackage[top=0.5in,margin=1.2in,includeheadfoot,bottom=1in]{geometry} 
\usepackage{lmodern}
\usepackage{hyperref}
\usepackage[capitalize]{cleveref}
\usepackage{mathrsfs}
\usepackage[norefs]{refcheck}

\usepackage{tikz-cd}
\usetikzlibrary{babel}

\usepackage{mathtools, tensor}
\usepackage{thmtools}

\newtheorem{theorem}{Theorem}[section]
\newtheorem{lemma}[theorem]{Lemma}
\newtheorem{proposition}[theorem]{Proposition}
\newtheorem{corollary}[theorem]{Corollary}

\theoremstyle{definition}
\newtheorem{definition}[theorem]{Definition}
\newtheorem{assumption}[theorem]{Assumption}

\theoremstyle{remark}
\newtheorem{remark}[theorem]{Remark}
\newtheorem{example}[theorem]{Example}
\usepackage[shortlabels]{enumitem}
\setlist[enumerate,1]{label=\textup{(\roman*)}}
\setlist[enumerate,2]{label=\textup{(\alph*)}}

\newcommand{\comment}[1]{} 

\DeclarePairedDelimiter{\abs}{\lvert}{\rvert}
\DeclarePairedDelimiter{\norm}{\lVert}{\rVert}
\DeclarePairedDelimiter\bra{\langle}{\rvert}
\DeclarePairedDelimiter\ket{\lvert}{\rangle}
\DeclarePairedDelimiterX\braket[2]{\langle}{\rangle}{#1\,\delimsize\vert\,\mathopen{}#2}

\newcommand{\ftimes}[2]{\mathbin{\tensor[_{#1}]{\mathbin{\times}}{_{#2}}}}

\newcommand{\C}{\mathbb{C}}
\newcommand{\N}{\mathbb{N}}
\newcommand{\Z}{\mathbb{Z}}
\newcommand{\R}{\mathbb{R}}

\newcommand{\Hils}{\mathcal{H}}
\newcommand{\Bounded}{\mathbb{B}}
\newcommand{\Compact}{\mathbb{K}}
\newcommand{\Mat}{\mathbb{M}}
\newcommand{\Schwartz}{\mathcal{S}}
\newcommand{\nlBM}{\overline{\Psi_\Gamma(Y,\partial Y)}}

\DeclareMathOperator{\id}{id}
\DeclareMathOperator{\interior}{int}
\DeclareMathOperator{\boundary}{{b}}
\DeclareMathOperator{\vol}{vol}
\DeclareMathOperator{\Image}{Im}
\DeclareMathOperator{\Ell}{Ell}

\DeclareMathOperator{\End}{End}
\DeclareMathOperator{\supp}{supp}
\DeclareMathOperator{\ev}{ev}
\DeclareMathOperator{\pr}{pr}

\DeclareMathOperator{\Prim}{Prim}
\DeclareMathOperator{\smooth}{sm}

\begin{document}

\begin{abstract}
We consider a smooth compact manifold with boundary, $M$,  embedded in
a smooth manifold of the same dimension on which an amenable group
$\Gamma$ acts by isometries. We do not assume $M$ to be invariant
under $\Gamma$. This results in a {\em partial action} of $\Gamma$ on
$M^\circ$: For $g\in \Gamma$ we let $M^\circ_g = g(M^\circ)\cap
M^\circ$ and obtain diffeomorphisms $g:M^\circ_{g^{-1}} \to
M^\circ_g$.

We assume that any two images of $\partial M$ under $ \Gamma$ either
coincide or are disjoint and that only finitely many lie in $M$. The
spherical blow-up of these images of $\partial M$ in $M$ yields a
manifold $Y$ with  boundary consisting of finitely many components.
Moreover, $Y$ inherits a partial action of~$\Gamma$.

We can then define the $C^*$-algebra $\mathcal
A=\overline{\Psi_\Gamma(Y,\partial Y)}$ of operators on $L^2(Y)\oplus
L^2(\partial Y)$, generated by the algebra $\Psi(Y,\partial Y)$ of
operators of order and type zero in Boutet de Monvel's calculus on $Y$
and partial isometries associated with the partial action. Denote by
$\Sigma=\overline{\Psi(Y,\partial Y)}/\Compact $ the symbol space.
If the partial action of $\Gamma$ on $\Prim(\Sigma)$ is topologically
free, we find a criterion for the Fredholm property of the operators
in $\overline{\Psi_\Gamma(Y,\partial Y)}$.

Moreover, we obtain the classification of the elliptic elements in
$\overline{\Psi_\Gamma(Y,\partial Y)}$ modulo stable homotopies:
For $\mathcal A_0= C(Y\sqcup \partial Y)\rtimes\Gamma$
$$\Ell(\mathcal A_0,\mathcal A)\cong K_0(C_0(T^*Y^\circ)\rtimes\Gamma
)\oplus K_0(C(\partial  Y)\rtimes \Gamma).$$
If $\Gamma$ is finitely generated and of polynomial growth, then the
elements associated with the second summand do not contribute to the
index. \\[.3ex]
\end{abstract} 
\maketitle
\tableofcontents

\section{Introduction}

Associated with the action \(\theta\) of a group \(\Gamma\) on a manifold \(M\) is a class of operators generated by differential operators and shift operators along the orbits. More precisely, operators in this class have the form of finite linear combinations
$$
D=\sum_{g \in \Gamma} D_g T_g: C^\infty(M)\longrightarrow C^\infty(M),
$$
where $\{D_g\}$ is a collection of differential operators on the manifold $M$, parametrized by elements of the group $\Gamma$, and $(T_g u)(x)=u(\theta_{g^{-1}}(x))$ is a shift operator. Operators of this type, hereinafter called {\em $\Gamma$-operators} for brevity, arise in non-commutative geometry and mathematical physics, see, e.g., \cite{Ant4,Con4,Con94,CoLa2,OnSk2}. The elliptic theory of such operators on smooth compact closed manifolds has been studied in sufficient detail. For example, conditions for the ellipticity of operators were obtained, the corresponding criteria for the Fredholm property were established~\cite{AL93,AnLe2,BSS26} and index formulas were  established \cite{NSS08,SaSt39,AbSa24,Per3,Per5}; generalizations to more general operators are considered in \cite{PeRo1,SaSch1,GKN1,SaSchr22}. The symbol algebras of such operators are {\em crossed products} of the symbol algebra of ordinary pseudodifferential operators with the group acting on this algebra by automorphisms. Therefore, the above results are closely related to results in the theory of crossed products; see, for example,~\cite{Ped1,BaCo2,AL93, Bla98}.

$\Gamma$-operators were also considered on manifolds with boundary; see \cite{Ant1, AnLe4}. More precisely, here it is assumed that $\Gamma$ acts on a manifold with boundary (in particular, it maps the boundary to itself). Instead of the algebra of pseudodifferential operators, the algebra of Boutet de Monvel matrix operators on a manifold with boundary is considered, see \cite{BM71,RS82,Gru96,S01}. The technique of crossed products also works in this situation. In the cited works, conditions for ellipticity and a criterion for the Fredholm solvability of problems have been established. For elliptic $\Gamma$-boundary value problems, a classification up to stable homotopies is obtained, and index formulas were shown in  \cite{SaSt45,BS21,BolSa3}. These results are generalizations of well-known results for the original Boutet de Monvel algebra~\cite{Fds17,MNS03,MSS06}.

However, on manifolds with boundary, a fundamentally different geometric situation arises and new classes of $\Gamma$-operators appear in applications (in plasma theory \cite{BiSo3},  sandwich shells and plates in engineering~\cite{OnSk1}). Namely, the manifold with boundary $M$ is a submanifold in  some ambient manifold $W$, the group $\Gamma$ acts on the ambient manifold $W$, and the submanifold $M\subset W$ is not invariant, i.e., orbits of some of the points in   $M$ are not contained in $M$ (see \cite{Sku16,Ross1,SaSt43,Sav18,BNSS22}). In this case, there is no longer a group action on the submanifold $M$, but a delicate geometric structure is defined, which is called a {\em partial group action}. Let \(X=M^\circ\) denote the interior of \(M\). For each element   $g\in \Gamma$, the set $X_g=\theta_g(X)\cap X$ is defined, and the element $g$ determines a bijection $\theta_g: X_{g^{-1}}\to X_g$. 

Our aim in this paper is to construct an elliptic theory for partial group actions on a manifold with boundary. We define a $C^*$-algebra of pseudodifferential boundary value problems with shifts associated  with a partial action, prove a criterion for the Fredholm solvability and, finally, obtain a classification of Fredholm boundary value problems modulo stable homotopies.  

Let us briefly describe the contents of the paper. We start by recalling in \cref{sect:PartialGroupAction} and~\cref{sect:GroupoidAlgebra} the necessary facts about partial actions of groups on topological spaces and $C^*$-algebras as well as the  theory of partial crossed products associated with a partial action developed in~\cite{Exe94, Mcc95, Exe98, Leb05,Exe17}.

We suppose that the action of  $\Gamma$ is isometric and the following regularity condition is satisfied: There are only finitely many {\em different}  sets $M_g$ and all of them are smooth manifolds with boundary. 

This regularity assumption is satisfied  for instance in the following setting. Consider the finite cylinder $M=\mathbb{S}^1\times [0,l]$, of length $l>1$, as a submanifold with boundary of the infinite cylinder $W=\mathbb{S}^1\times\mathbb{R}$, on which the group $\Gamma=\mathbb Z$ acts by shifts  $\theta_k:(x,t)\mapsto (x,t+k)$ for $k\in \mathbb Z $.

%

In \cref{sec:closed} we show that given an  action of a group $\Gamma$ on a manifold $W$  and an arbitrary compact submanifold with boundary $M\subset W$ of codimension zero,  we have an induced partial action of $\Gamma$ on the interior $X=M^\circ$. 

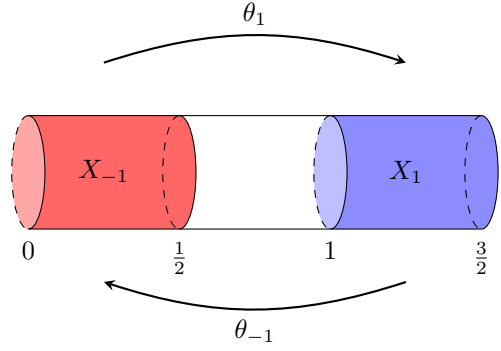
\begin{figure}[hbt!]
	\begin{tikzpicture}[scale=1, >=stealth]
		
		\def\R{0.75}
		\def\L{6}
		\def\a{2}
		\def\b{4}
		\def\er{0.22}
		
		\newcommand{\ellipseThreeD}[1]{%
			\draw[dashed] (#1,\R) arc[start angle=90, end angle=270,
			x radius=\er, y radius=\R];
			\draw (#1,-\R) arc[start angle=270, end angle=450,
			x radius=\er, y radius=\R];
		}
		
		\fill[red!60]  (0,-\R) rectangle (\a,\R);
		\fill[blue!45] (\b,-\R) rectangle (\L,\R);
		
		\fill[red!35]  (0,0) ellipse[x radius=\er, y radius=\R];
		\fill[red!60]  (\a,0) ellipse[x radius=\er, y radius=\R];
		\fill[blue!25] (\b,0) ellipse[x radius=\er, y radius=\R];
		\fill[blue!45] (\L,0) ellipse[x radius=\er, y radius=\R];
		
		\draw (0,\R) -- (\L,\R);
		\draw (0,-\R) -- (\L,-\R);
		
		\foreach \x/\lab in {
			0/{0},
			\a/{\frac12},
			\b/{1},
			\L/{\frac32}
		}{
			\ellipseThreeD{\x}
			\node[below] at (\x,-0.8) {$\lab$};
		}
		
		\node at (1,0) {$X_{-1}$};
		\node at (5,0) {$X_1$};
		
		\draw[->, thick, bend left=18]
		(1,1.45)
		to node[above] {$\theta_1$}
		(5,1.45);
		
		\draw[->, thick, bend left=18]
		(5,-1.45)
		to node[below] {$\theta_{-1}$}
		(1,-1.45);
		
	\end{tikzpicture}
	\caption{Partial group action of \(\Z\) on the open cylinder \(X=S^1\times(0,\tfrac32)\) }
\end{figure}

We introduce a regularity condition (Assumption~\ref{ass:regularity1}) under which the set $\Gamma(\partial M)\cap M$  is a smooth closed submanifold in $M$ of codimension one.  We denote by $Y$ the spherical blowup of $M$ along $\Gamma(\partial M)\cap M$. This is a compact manifold with boundary obtained by cutting $M$ along the submanifold $\Gamma(\partial M)\cap M$). Then the partial action of $\Gamma$ on $M$ extends by continuity to a partial action on $Y$. 

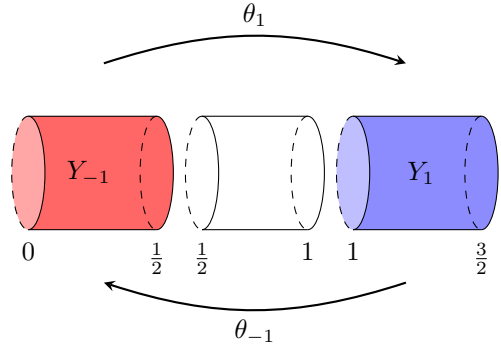
\begin{figure}[hbt!]
	\begin{tikzpicture}[scale=1, >=stealth]
		
		\def\R{0.75}
		\def\L{6}
		\def\a{2}
		\def\b{4}
		\def\er{0.22}
		\def\gap{0.3}
		
		\newcommand{\ellipseThreeDCut}[1]{%
			\draw[dashed] (#1,\R) arc[start angle=90, end angle=270,
			x radius=\er, y radius=\R];
			\draw (#1,-\R) arc[start angle=270, end angle=450,
			x radius=\er, y radius=\R];
		}
		
		\fill[red!60] (0,-\R) rectangle (\a-\gap,\R);
		\fill[red!35] (0,0) ellipse[x radius=\er, y radius=\R];
		\fill[red!60] (\a-\gap,0) ellipse[x radius=\er, y radius=\R];
		
		\fill[white] (\a+\gap,-\R) rectangle (\b-\gap,\R);
		\fill[white] (\a+\gap,0) ellipse[x radius=\er, y radius=\R];
		\fill[white] (\b-\gap,0) ellipse[x radius=\er, y radius=\R];
		
		\fill[blue!45] (\b+\gap,-\R) rectangle (\L,\R);
		\fill[blue!25] (\b+\gap,0) ellipse[x radius=\er, y radius=\R];
		\fill[blue!45] (\L,0) ellipse[x radius=\er, y radius=\R];
		
		\draw (0,\R) -- (\a-\gap,\R);
		\draw (0,-\R) -- (\a-\gap,-\R);
		\ellipseThreeDCut{0}
		\ellipseThreeDCut{\a-\gap}
		
		\draw (\a+\gap,\R) -- (\b-\gap,\R);
		\draw (\a+\gap,-\R) -- (\b-\gap,-\R);
		\ellipseThreeDCut{\a+\gap}
		\ellipseThreeDCut{\b-\gap}
		
		\draw (\b+\gap,\R) -- (\L,\R);
		\draw (\b+\gap,-\R) -- (\L,-\R);
		\ellipseThreeDCut{\b+\gap}
		\ellipseThreeDCut{\L}
		
		\node[below] at (0,-0.8) {$0$};
		
		\node[below] at (\a-\gap,-0.8) {$\frac12$};
		\node[below] at (\a+\gap,-0.8) {$\frac12$};
		
		\node[below] at (\b-\gap,-0.8) {$1$};
		\node[below] at (\b+\gap,-0.8) {$1$};
		
		\node[below] at (\L,-0.8) {$\frac32$};
		
		\node at (0.8,0) {$Y_{-1}$};
		\node at (5.2,0) {$Y_1$};
		
		\draw[->, thick, bend left=18]
		(1,1.45)
		to node[above] {$\theta_1$}
		(5,1.45);
		
		\draw[->, thick, bend left=18]
		(5,-1.45)
		to node[below] {$\theta_{-1}$}
		(1,-1.45);
		
	\end{tikzpicture}
	\caption{Partial group action of \(\Z\) on the cut cylinder \(Y=S^1\times\left([0,\tfrac12]\sqcup [\tfrac12, 1]\sqcup [1,\tfrac32]\right)\)}
\end{figure}

In \cref{sect:BdM}, we consider $\Gamma$-boundary value problems on  $Y$: We suppose that our partial action on $Y$ is isometric and define the $C^*$-algebra $\nlBM\subset \Bounded(L^2(Y)\oplus L^2(\partial Y))$  as the $C^*$-algebra generated by the Boutet de Monvel algebra of operators of order and type zero on $Y$ denoted by $\overline{\Psi(Y,\partial Y)}$ and the shift operators associated with the partial group action on $Y$ and $\partial Y$. The Fredholm property of operators in this algebra is equivalent to the invertibility in the Calkin algebra $\nlBM/{\Compact}$, where $\Compact$ is the ideal of compact operators.  To describe this Calkin algebra, we recall the symbol isomorphism
$$
\overline{\Psi(Y,\partial Y)}/\Compact\stackrel{\sigma}\longrightarrow \Sigma\subset
C(S^*Y) \oplus C(S^*(\partial Y),\End(L^2(N\partial Y_+)\oplus\C))
$$
for the classical Boutet de Monvel algebra. Here $\Sigma$ is the $C^*$-algebra of principal symbols. Then the partial action of $\Gamma$ on $Y$ naturally extends to a partial action on the $C^*$-algebra $\Sigma$ and the partial crossed product $\Sigma\rtimes \Gamma$ is defined.  Under the assumption that $\Gamma$ acts topologically freely on the primitive ideal space $\Prim \Sigma$ (see~\cite{Leb05}) and $\Gamma$ is amenable,  we construct a $C^*$-algebra isomorphism
$$
\nlBM/\Compact \stackrel{\widetilde{\sigma}}\longrightarrow \Sigma\rtimes\Gamma.
$$ 
This isomorphism gives the corresponding criterion for the Fredholm property. 

\cref{sec:partial-BM} and \cref{sect:Ktheory}  are devoted to obtaining the classification of elliptic $\Gamma$-boundary value problems modulo stable homotopies. Here we consider two algebras $\mathcal{A}_0\subset \mathcal{A}$
$$
\mathcal{A}=\nlBM,\qquad  \mathcal{A}_0=C(Y\sqcup \partial Y)\rtimes\Gamma.
$$
Following \cite{Sav05},  we consider in \cref{sect:homotopy}  the Abelian group  \(\Ell(\mathcal A_0,\mathcal A)\) generated by 
triples $(D,P,Q)$, where $P,Q$ are matrix projections over $\mathcal{A}_0$, while $D$ is a matrix operator over $\mathcal A$ that is elliptic as an operator acting between the ranges of $P$ and $Q$.  These triples are considered modulo stable homotopies.  This  group is naturally isomorphic to the $K_0$-group of the mapping cone of the inclusion $\mathcal{A}_0\subset \mathcal{A}/\mathbb{K}$. We then use the $K$-theoretic approach developed in~\cite{MNS03,MSS06} and obtain the following isomorphism
\begin{equation*}
	\Ell(\mathcal A_0,\mathcal{A})\simeq 	K_0(C_0(T^*Y^\circ)\rtimes\Gamma)\oplus K_0(C(\partial Y)\rtimes\Gamma),
\end{equation*}
which is the main result of this paper. Here elements in the $\Ell$-group corresponding to the first summand have representatives $(D,P,Q)$, where $D$ is trivial near the boundary $\partial Y$ (more precisely, near the boundary it  is a matrix operator over $\mathcal{A}_0$). If, in addition, $\Gamma$ is finitely generated and of polynomial growth, we show that  elements corresponding to the second summand have as representatives special $\Gamma$-boundary value problems of index zero and thus do not contribute to the index.  This result is similar to Boutet de Monvel's index theorem, see~\cite{BM71,RS82}.

\section{Partial group actions and their groupoids}\label{sect:PartialGroupAction}
	In the following, \(\Gamma\) is a countable discrete group which acts on a locally compact Hausdorff space~\(W\) by homeomorphisms. Let \(\alpha\colon\Gamma\curvearrowright W\) denote the action and write \(\alpha_g(x)=g\cdot x\).
	\begin{definition}
		The \emph{action groupoid} of \(\Gamma\curvearrowright W\) denoted by \(\Gamma\ltimes W\) is as a space \(\Gamma\times W\) with unit space \((\Gamma\ltimes W)^{(0)}=W\). The range and source maps \(r,s\colon \Gamma\ltimes W\to W\) are defined by 
		\begin{equation*}
			s(g,x)=x, \qquad r(g,x)=g\cdot x.
		\end{equation*}
		The groupoid multiplication is the map \((\Gamma\ltimes W)^{(2)}=(\Gamma\ltimes W)\ftimes{s}{r}(\Gamma\ltimes W)\to\Gamma\ltimes W\) given by \begin{equation*}
			(h,g\cdot x)\cdot(g,x)=(hg,x).
		\end{equation*}
		The inversion \(\Gamma\ltimes W\to \Gamma\ltimes W\) is given by \((g,x)^{-1}=(g^{-1},g\cdot x)\).
	\end{definition}
	It is a topological groupoid with respect to the product topology on \(\Gamma\ltimes W\).
	Now, we assume that \(X\subset W\) is an open subset, not necessarily invariant under \(\Gamma\).  
	Then the groupoid structure on \(\Gamma\ltimes W\) restricts to the subset
	\begin{equation*}
		(\Gamma\ltimes W)^X_X=\{(g,x)\in\Gamma\ltimes W\colon r(g,x),s(g,x)\in X\}=\{(g,x)\in\Gamma\times X\colon g\cdot x\in X\}.
	\end{equation*}
	We shall use the notation \(\Gamma \ltimes X\) for this groupoid, even though it is not an action groupoid unless \(X\) is \(\Gamma\)-invariant. 
	
	The setting also fits in the framework of partial group actions, see \cites{Exe98,Mcc95}.
	\begin{definition}
		Let \(\Gamma\) be a countable discrete group and \(X\) a locally compact Hausdorff space. Then a partial action of \(\Gamma\) on \(X\) is given by a family of open subsets \((X_g)_{g\in\Gamma}\) and homeomorphisms \((\theta_g\colon X_{g^{-1}}\to X_{g})_{g\in\Gamma}\) such that
		\begin{enumerate}
			\item \(X_e=X\) and \(\theta_e=\id\),
			\item\label{item:images} \(\theta_g(X_{g^{-1}}\cap X_h)\subseteq X_g\cap X_{gh}\) for all \(g,h\in\Gamma\),
			\item \(\theta_g(\theta_h(x))=\theta_{gh}(x)\) for all \(g,h\in \Gamma\) and \(x\in X_{h^{-1}}\cap X_{h^{-1}g^{-1}}\).
		\end{enumerate}
	\end{definition}
Note that the axioms imply that \ref{item:images} is actually an equality of sets. 
\begin{lemma}\label{res:action-restricts-to-partial-action}
	Let \(\alpha\colon \Gamma\curvearrowright W\) be a group action of a discrete group on a locally compact Hausdorff space by homeomorphisms and \(X\subseteq W\) open. Let \(X_g =\{x\in X\colon g^{-1}\cdot x\in X\}\) and let \(\theta_g\) be the restriction of the homeomorphism \(\alpha_g\colon W\to W\) to \(X_{g^{-1}}\to X_{g}\). Then this defines a partial group action of \(\Gamma\) on \(X\).
\end{lemma}
We give some examples, to which we frequently come back in the following.

\begin{example}\label{ex:shifts}
	Let \(W=S^1\times\R\) be an infinite cylinder and let \(\Gamma=\Z\) act on it by shifts \(\alpha_k(x,t)=(x,t+k)\) for \(k\in\Z\) and \((x,t)\in S^1\times\R\). Consider the finite cylinder \(X=S^1\times(0,\tfrac{3}{2})\). By Lemma \ref{res:action-restricts-to-partial-action},  \(X_{-1}=S^1\times(0,\tfrac{1}{2})\), \(X_0=X\), \(X_1=S^1\times(1,\tfrac{3}{2})\) and \(X_k=\emptyset\) for all \(k\notin\{-1,0,1\}\).
\end{example}

\begin{example}\label{ex:shifts-var-2}
	As a variant of \cref{ex:shifts} let \(\Gamma=\Z_2\times\Z\) act on \(W=S^1\times\R\) by \(\alpha_{(\overline p, k)}(x,t)=(e^{i\pi p}x,t+k)\) for \(k\in\Z\), \(\overline p\in\{\overline 0,\overline 1\}\) and \((x,t)\in S^1\times\R\). Consider the finite cylinder \(X=S^1\times(0,\tfrac{3}{2})\). By Lemma \ref{res:action-restricts-to-partial-action},  
\(X_{\overline 0, -1}=S^1\times (0,\tfrac{1}{2})= X_{\overline 1, -1}\), \(X_{\overline 0,0}=X=X_{\overline1,0}\), \(X_{\overline 0,1}=S^1\times (1,\tfrac{3}{2})=X_{\overline 1,1}\) and \(X_{\overline p, k}=\emptyset\) for all \((\overline p,k)\notin\{\overline 0, \overline 1\}\times\{-1,0,1\}\).
\end{example}

\begin{example}[\cite{Sav18}]\label{ex:disk_skalings} Let \(W=\R^2\) and \(\Gamma=\Z\) act by scalings \(\alpha_k(x)=2^{-k}x\) for \(k\in\Z\) and \(x\in\R^2\). Let \(X=\mathbb D=\{x\in\R^2\colon \norm{x}<1\}\) be the open unit disk. Then \(X_k=\{x\in\R^2\colon \norm{x}<2^{-k}\}\) for \(k\in\N_0\) and \(X_k=X\) for \(k\in-\N\).
\end{example}
For a partial action of \(\Gamma\) on \(X\) one can define a topological groupoid, which coincides with \((\Gamma\ltimes W)^X_X\) in the situation of \cref{res:action-restricts-to-partial-action}. Namely, let
\[\Gamma\ltimes X =\{(g,x)\in \Gamma\times X\colon x\in X_{g^{-1}}\}\] be equipped with the range and source maps \(r,s\colon\Gamma\ltimes X\to X\) defined by \(s(g,x)=x\) and \(r(g,x)=\theta_g(x)\). The multiplication \((h,\theta_g(x))(g,x)=(hg,x)\) and inversion \((g,x)^{-1}=(g^{-1},\theta_g(x))\) are well-defined by the properties of partial actions.
\begin{remark}\label{rem:etale}
	Note that the groupoid \(\Gamma\ltimes X\) is \'{e}tale for the subspace topology of \(\Gamma\ltimes X\) in \(\Gamma\times X\). As \(\Gamma\) is discrete, for each \((g,x)\in \Gamma\ltimes X\) the source map \(s\) restricts to a homeomorphism between the open subsets \(\{g\}\times X_{g^{-1}}\) of \(\Gamma\ltimes X\) and \(X_{g^{-1}}\) of \(X\).
\end{remark}

More generally, there is a notion of partial group actions on \(C^*\)-algebras.
\begin{definition}[\cite{Mcc95}]
	Let \(\Gamma\) be a discrete group and \(A\) a \(C^*\)-algebra. Then a partial action of \(\Gamma\) on \(A\) is given by a family of closed ideals \((A_g)_{g\in\Gamma}\) and \(^*\)-isomorphisms \((\theta_g\colon A_{g^{-1}}\to A_{g})_{g\in\Gamma}\) such that
	\begin{enumerate}
		\item \(A_e=A\) and \(\theta_e=\id\),
		\item \(\theta_g(A_{g^{-1}}\cap A_h)\subseteq A_g\cap A_{gh}\) for all \(g,h\in\Gamma\),
		\item \(\theta_g(\theta_h(a))=\theta_{gh}(a)\) for all \(g,h\in \Gamma\) and \(a\in A_{h^{-1}}\cap A_{h^{-1}g^{-1}}\).
	\end{enumerate}
\end{definition}
\begin{definition}\label{def:equivariant_homomorphism}
	Let \(A,B\) be \(C^*\)-algebras with partial actions \(\theta\colon\Gamma\curvearrowright A\) and \(\tau\colon\Gamma\curvearrowright B\). Then a \(^*\)-homomorphism \(\varphi\colon A\to B\) is called \emph{\(\Gamma\)-equivariant} if \(\varphi|_{A_g}\colon A_g\to B_g\) for all \(g\in \Gamma\) and \(\varphi(\theta_g(a))=\tau_g(\varphi(a))\) for all \(g\in \Gamma\) and \(a\in A_{g^{-1}}\).
\end{definition}
\begin{lemma}\label{res:partial-action-on-pullback}
	Let \(A,B,C\) be \(C^*\)-algebras with partial \(\Gamma\)-actions and suppose \(\varphi\colon A\to C\) and \(\psi\colon B\to C\) are \(\Gamma\)-equivariant \(^*\)-homomorphisms. Then the pullback \(A\oplus_C B\) admits a partial \(\Gamma\)-action with \((A\oplus_C B)_g=A_g\oplus_{C_g}B_g\) and \(g\cdot(a,b)=(g\cdot a,g\cdot b)\) for all \(g\in G\) and \((a,b)\in A\oplus_B C\).
\end{lemma}
This is straight forward to check. One particular example is the mapping cone.
\begin{example}\label{ex:mapping-cone}
	Suppose \(f\colon A\to B\) is a \(^*\)-homomorphism. The the mapping cone \begin{equation*}C_f = \{(a,b)\in A\oplus C_0((0,1], B)\colon f(a)=b(1)\}\end{equation*}
	is the pullback with respect to the maps \(f\colon A\to B\) and \(\ev_1\colon C_0((0,1],B)\to B\). Similarly, the mapping cylinder
	\begin{equation*}Z_f = \{(a,b)\in A\oplus C([0,1], B)\colon f(a)=b(1)\}\end{equation*}
	is a pullback.
	
	When \(A\) and \(B\) are equipped with partial \(\Gamma\)-actions and \(f\colon A\to B\) is \(\Gamma\)-equivariant, define a partial \(\Gamma\)-action on \(C_0((0,1], B)\) by \(C_0((0,1], B)_g=C_0((0,1], B_g)\) and \((g\cdot b)(t)=g\cdot b(t)\). Then \(\ev_1\) is \(\Gamma\)-equivariant and \(C_f\) has a partial \(\Gamma\)-action. Similarly, \(Z_f\) admits a partial \(\Gamma\)-action.
\end{example}
\section{Groupoid \(C^*\)-algebra and crossed product \(C^*\)-algebra}\label{sect:GroupoidAlgebra}
For a partial group action of \(\Gamma\) on \(X\), one can define the groupoid \(C^*\)-algebra of \(\Gamma\ltimes X\). Alternatively, in \cite{Mcc95} there is a construction of a crossed product \(C^*\)-algebra for a partial \(\Gamma\)-action on a \(C^*\)-algebra. In this section, we recall both constructions and that they give isomorphic \(C^*\)-algebras for partial actions on spaces. Moreover, we list some functorial properties needed later on.
\subsection{Groupoid \(C^*\)-algebra}
Suppose there is a partial \(\Gamma\)-action on a locally compact Hausdorff space \(X\).
For \(x\in X\) the source fiber of the groupoid \(\Gamma\ltimes X\) is \(s^{-1}\{x\}=\{(g,x)\colon x\in X_{g^{-1}}\}\). We equip each source fiber with the counting measure. This gives rise to a right Haar system on \(\Gamma\ltimes X\), so that its groupoid \(C^*\)-algebra can be defined, see \cite{Ren80}. 
\begin{definition}
	Equip \(C_c(\Gamma\ltimes X)\) with the following involution and convolution
	\begin{align*}
		f^*(g,x)&=\overline{f(g^{-1},g\cdot x)},\\
		f_1*f_2(g,x)&=\sum_{\substack{h\in\Gamma\\ x\in X_{h^{-1}}}}f_1(gh^{-1},h\cdot x)f_2(h,x).
	\end{align*}
\end{definition}
Then \(C^*(\Gamma\ltimes X)\) is defined as the completion of \(C_c(\Gamma\ltimes X)\) with respect to all \(^*\)-representations bounded by the \(I\)-norm. The \(I\)-norm is defined as \(\norm{f}_I=\max\{\norm{f}_s,\norm{f^*}_s\}\) where
\begin{equation*}
	\norm{f}_s=\sup_{x\in X} \sum_{\substack{g\in\Gamma\\ x\in X_{g^{-1}}}}\abs{f(g,x)}.
\end{equation*}
\subsection{Crossed products for partial group action}\label{sec:partial-crossed-products}
There is also a notion of crossed product for partial group actions on \(C^*\)-algebras \cite{Mcc95}, we recall now.  Suppose there is a partial \(\Gamma\)-action on \(A\) by \(*\)-automorphisms \(\theta_g\), \(g\in \Gamma\).
Let \(L(A,\Gamma)\) be the closed subspace of \(\ell^1(\Gamma,A)\) defined by
\[L(A,\Gamma)=\{f\in\ell^1(\Gamma,A)\colon f(g)\in A_g \text{ for all }g\in \Gamma\}.\]
It becomes a Banach \(^*\)-algebra when equipped with the following involution and convolution 
\begin{align*}
	f^*(g)&=\theta_g(f(g^{-1}))^*,\\
	f_1*f_2(g)&=\sum_{h\in\Gamma}\theta_h[\theta_{h^{-1}}(f_1(h))f_2(h^{-1}g)].
\end{align*}
\begin{definition}[\cite{Mcc95}]
	A \emph{covariant representation} of \((A,\Gamma,\theta)\) is a triple \((\pi,u,\Hils)\) consisting of a Hilbert space \(\Hils\), a non-degenerate \(^*\)-representation \(\pi\colon A\to \Bounded(\Hils)\) and a map \(u\colon \Gamma\to\Bounded(\Hils)\) such that for all \(g,h\in\Gamma\)
	\begin{enumerate}
		\item \(u_g\) is a partial isometry with initial space \(\overline{ \pi(A_{g^{-1}})\Hils}\) and final space \(\overline{\pi(A_g)\Hils}\),
		\item \label{item:covariant}\(u_g\pi(a)u_{g^{-1}}=\pi(\theta_g(a))\) for all \(a\in A_{g^{-1}}\),
		\item \(\pi(a)(u_gu_h-u_{gh})=0\) for all \(a\in A_g\cap A_{gh}\),
		\item \(u_g^*=u_{g^{-1}}\).
	\end{enumerate}
\end{definition}
\begin{lemma}[\cite{Mcc95}*{p.~84}]\label{res:unitary-rep-gives-cov-rep}
	Suppose \(U\colon \Gamma\to\mathcal U(\Hils)\) is a unitary representation and that \(\pi\colon A\to \Bounded(\Hils)\) a \(^*\)-representation satisfying \(U_g\pi(a)U_{g^{-1}}=\pi(\theta_g(a))\) for all \(a\in A_{g^{-1}}\). Let \(P_g\) denote the orthogonal projection to \(\overline{ \pi(A_g)\Hils}\) for \(g\in\Gamma\) and set \(u_g=P_gU_g=U_gP_{g^{-1}}\). Then \((\pi,u,\Hils)\) is a covariant representation of \((A,\Gamma,\theta)\).
 \end{lemma}
\begin{example}\label{ex:covariant-rep}
	Suppose \(W\) is equipped with a finite \(\Gamma\)-invariant measure \(\mu\) and \(X\) is measurable with \(\mu(X)>0\). Then there is an associated covariant representation of \((C_0(W),\Gamma,\alpha)\) defined by \((\pi^W,u^W,L^2(W,\mu))\) where \(\pi^W\) is the representation by multiplication operators and \(u^W\colon \Gamma\to \mathcal U(L^2(W,\mu))\) is defined by \(u^W(g)\psi(x)=\psi(g^{-1}\cdot x)\).
	
	Equip \(X\) now with the restriction of \(\mu\) and let \(\pi^X\) denote the representation of \(C_0(X)\) by multiplication on \(L^2(X,\mu)\). Define \(u^X\colon \Gamma\to \mathcal \Bounded(L^2(X,\mu))\) by 
	\begin{align*}
		u^X_g\psi(x)=\begin{cases}
			\psi(g^{-1}\cdot x) &\text{if }g^{-1}\cdot x\in X,\\
			0 &\text{else.}
		\end{cases}
	\end{align*}
	Then using \cref{res:unitary-rep-gives-cov-rep}, one can show that \((\pi^X,u^X,L^2(X,\mu))\) is a covariant representation of \((C_0(X),\Gamma,\theta)\). Let \(e\colon L^2(X,\mu)\to L^2(W,\mu)\) and \(r\colon L^2(W,\mu)\to L^2(X,\mu)\) denote extension by zero and restriction. Then the following diagrams commute for all \(f\in C_0(X)\) and \(g\in \Gamma\)
	\begin{eqnarray*}
		\begin{tikzcd}
			L^2(W,\mu)\arrow[r,"\pi^W(f)"] & L^2(W,\mu) \arrow[d,"r"]\\
			L^2(X,\mu)\arrow[u,"e"]\arrow[r,"\pi^X(f)"] & L^2(X,\mu)
		\end{tikzcd}\qquad 
		\begin{tikzcd}
			L^2(W,\mu)\arrow[r,"u^W_g"] & L^2(W,\mu) \arrow[d,"r"]\\
			L^2(X,\mu)\arrow[u,"e"]\arrow[r,"u^X_g"] & L^2(X,\mu)
			\end{tikzcd}.
	\end{eqnarray*} 
\end{example}
\begin{definition}
	Every covariant representation \((\pi,u,\Hils)\) of \((A,\Gamma,\theta)\) gives rise to a \(^*\)-representation of \(L(A,\Gamma)\) on \(\Hils\) defined by
	\begin{equation*}
		(\pi\times u)(f)=\sum_{g\in\Gamma}\pi(f(g))u_g.
	\end{equation*}
	Then \(A\rtimes\Gamma\) is defined as the \(C^*\)-completion of \(L(A,\Gamma)\) with respect to all covariant representations.
\end{definition}
By \cite{Mcc95}*{Proposition~2.8} there is a \(1\)-\(1\) correspondence between covariant representations and nondegenerate representations of \(L(A,\Gamma)\).
\begin{remark}
	There is also a notion of reduced crossed products for partial group actions. If the group \(\Gamma\) is amenable, the reduced crossed product is isomorphic to \(A\rtimes \Gamma\), see \cite{Mcc95}*{Proposition~4.2}.
\end{remark}
\begin{remark}\label{rem:fell-bundle}
	A partial \(\Gamma\)-action on a \(C^*\)-algebra \(A\) given by \((A_g)_{g\in\Gamma}\) and \((\theta_g)_{g\in \Gamma}\) gives rise to a Fell bundle, see \cite{Exe17}*{16.5}. It is given by \((B_g)_{g\in \Gamma}\) with \(B_g=\{a\delta_g\colon a\in A_g\}\) and multiplication maps \(B_g\times B_h\to B_{gh}\) given by \((a\delta_g)\cdot(b\delta_h)=\theta_g(\theta_{g^{-1}}(a)b)\delta_{gh}\). The involution maps \(B_g\to B_{g^{-1}}\) are given by \((a\delta_g)^*=\theta_{g^{-1}}(a^*)\delta_{g^{-1}}\). Then \(A\rtimes\Gamma\) is isomorphic to the \(C^*\)-algebra of the Fell bundle, see \cite{Exe17}*{Proposition~16.28}.
\end{remark}
\begin{theorem}\cite{Aba04}*{Theorem~3.3}
	Suppose \(X\) admits a partial \(\Gamma\)-action. Then \(C^*(\Gamma\ltimes X)\cong  C_0(X)\rtimes\Gamma\) for the induced \(\Gamma\)-action on \(C_0(X)\). 
\end{theorem}

\subsection{Topologically free actions and faithful representations}
In the following, it will be important to know when a covariant representation \((\pi,u,\Hils)\) gives rise to a faithful representation of the corresponding crossed product. 
\begin{definition}[\cite{Leb05}]
	A partial group action of \(\Gamma\) on \(X\) is called \emph{topologically free} if for every finite set \(\{g_1,\ldots,g_k\}\subset\Gamma\setminus\{e\}\) the union of the sets fixed by \(g_i\)
	\begin{equation*}
		\bigcup_{j=1}^n\{x\in X_{g_j^{-1}}\colon \theta_{g_j}(x)=x\}
	\end{equation*}
	has empty interior. 
\end{definition}
The partial action in our two examples above are topologically free. The shifts in \cref{ex:shifts}, \cref{ex:shifts-var-2} have no fixed points and the scalings in \cref{ex:disk_skalings} have only \(\{0\}\) as fixed point sets.

When \(\theta\) is a partial \(\Gamma\)-action on a \(C^*\)-algebra \(A\), there are induced partial actions on its spectrum \(\widehat A\) and primitive spectrum \(\Prim A\). Namely, set for \(g\in \Gamma\)
\begin{equation*}
	(\widehat A)_g= \{[\pi]\in\widehat A\colon \pi(A_g)\neq 0\}\cong \widehat{A_g}
\end{equation*}
and \(g\cdot\pi(x)=\pi(\theta_{g^{-1}}(x))\) for \(\pi\in \widehat{A_{g^{-1}}}\) and \(x\in A_g\). The action on \(\Prim A\) is determined by \(\Prim(A)_g=\Prim(A_g)\) and \(g\cdot(\ker\pi)=\ker(g\cdot\pi)\). 
As in the case of a group action one has the following isomorphism theorem.
\begin{theorem}[\cite{Leb05}*{Corollary~3.8}]\label{res:isomorphism-theorem}
	Let \(\Gamma\) be an amenable, discrete group. Suppose that there is a partial action of \(\Gamma\) on \(A\) such that the induced partial action on \(\Prim A\) is topologically free. Let \((\pi, u,\Hils)\) be a covariant representation of \((A,\Gamma,\theta)\). Then the representation \(\pi\times u\) of \(A\rtimes\Gamma\) is faithful if and only if \(\pi\) is a faithful representation of \(A\).
\end{theorem}
\subsection{Functoriality}\label{subsec:functoriality}
The crossed product for partial actions is functorial in the following sense.
\begin{lemma}[\cite{Exe17}*{Proposition~22.2}]\label{res:induced-map-crossed-product}
	Let \(A,B\) be \(C^*\)-algebras with partial actions \(\theta\colon\Gamma\curvearrowright A\) and \(\tau\colon\Gamma\curvearrowright B\). Suppose \(\varphi\colon A\to B\) is a \(\Gamma\)-equivariant \(^*\)-homomorphism as defined in \cref{def:equivariant_homomorphism}. Then there is an induced \(^*\)-homomorphism
	\begin{equation*}
		\widetilde{\varphi}\colon A\rtimes\Gamma\to B\rtimes \Gamma 
	\end{equation*}
	which is given by \((\widetilde{\varphi} (f))(g)=\varphi(f(g))\) for \(f\in L(A,\Gamma)\). The construction is functorial, i.e. \(\widetilde\id =\id\) and if \(\varphi\colon A\to B\) and \(\psi\colon B\to C\) are \(\Gamma\)-equivariant, one has \(\widetilde{\psi\circ\varphi}=\widetilde\psi\circ\widetilde\varphi\).
\end{lemma}

In particular, this can be applied to continuous proper maps \(\varphi\colon X\to Y\) between locally compact Hausdorff spaces \(X,Y\) with partial \(\Gamma\)-actions \(\theta,\tau\) respectively, if \(\varphi|_{X_g}\colon X_g\to Y_g\) for all \(g\in \Gamma\) and \(\varphi(\theta_g(x))=\tau_g(\varphi(x))\) for all \(g\in \Gamma\) and \(x\in X_{g^{-1}}\). In this case, we denote the induced map on the crossed products by \(\varphi^*\colon C_0(Y)\rtimes \Gamma\to C_0(X)\rtimes \Gamma\).

Let \(A\) be a \(C^*\)-algebra with a partial \(\Gamma\)-action and \(I\subseteq A\) an ideal. Then \(I\) is called \(\Gamma\)-invariant if \(\theta_g(I\cap A_{{g}^{-1}})\subseteq I\) for all \(g\in\Gamma\). In this case, letting  \(I_g=I\cap A_g\) and restricting \(\theta_g\) to a map between \(I_{g^{-1}}\) and \(I_g\) for \(g\in \Gamma\) yields a partial \(\Gamma\)-action. Moreover, there is also a partial \(\Gamma\)-action on the quotient \(A/I\), where \((A/I)_g=A_g/I_g\) and \(\overline{\theta}_g([a])=[\theta_g(a)]\).
\begin{proposition}[\cite{ELQ02}*{Proposition~3.1}]\label{res:ses-partial-actions-algebras}
	Let \(A\) be a \(C^*\)-algebra with a partial \(\Gamma\)-action and \(I\subseteq A\) a \(\Gamma\)-invariant ideal. Denote by \(\iota\colon I\to A\) the inclusion and by \(\pi\colon A\to A/I\) the quotient map. Then the following sequence is exact
	\begin{equation*}
		\begin{tikzcd}
			0\arrow[r,""]&I\rtimes\Gamma\arrow[r,"\tilde\iota"]&A\rtimes \Gamma\arrow[r,"\tilde \pi"] &A/I\rtimes\Gamma\arrow[r,""]&0.
		\end{tikzcd}
	\end{equation*}
\end{proposition}
Restricting to partial actions on spaces, a subset \(U\subseteq X\) is called \(\Gamma\)-invariant if \(\theta_g(U\cap X_{g^{-1}})\subseteq U\) for all \(g\in\Gamma\). In this case, \(U\) has a partial \(\Gamma\)-action defined by \(U_g=U\cap X_g\) for \(g\in \Gamma\) and restricting \(\theta_g\) appropriately. One deduces from \cref{res:ses-partial-actions-algebras} the following result.
\begin{lemma}\label{res:ses-partial-actions}
	Let \(X\) be a locally compact Hausdorff space with a partial action by a discrete group \(\Gamma\) and let \(i\colon U\hookrightarrow X\) be an open \(\Gamma\)-invariant subset. Then \(X\setminus U\) is also \(\Gamma\)-invariant and there is a short exact sequence
	\begin{equation*}
		\begin{tikzcd}
			0\arrow[r,""]&C_0(U)\rtimes\Gamma\arrow[r,"i^*"]&C_0(X)\rtimes \Gamma\arrow[r,""] &C_0(X\setminus U)\rtimes\Gamma\arrow[r,""]&0.
		\end{tikzcd}
	\end{equation*}
\end{lemma}
%

The following adapts the arguments in the proof of \cite{Ped99}*{Thm.~6.3} to the setting of partial group actions.
\begin{proposition}\label{res:partial_cossed-fibred-product}
	Let \(A,B,C\) be \(C^*\)-algebras with partial \(\Gamma\)-actions and suppose \(\varphi\colon A\to C\) and \(\psi\colon B\to C\) are \(\Gamma\)-equivariant \(^*\)-homomorphisms. Equip \(A\oplus_C B\) with the partial \(\Gamma\)-action as in \cref{res:partial-action-on-pullback}. Then there is an isomorphism
	\begin{equation*}
		(A\oplus_C B)\rtimes\Gamma\cong (A\rtimes\Gamma)\oplus_{C\rtimes\Gamma}(B\rtimes\Gamma).
	\end{equation*}
\end{proposition}
\begin{proof}
	As the maps \(\pr_1\colon A\oplus_CB\to A\) and \(\pr_2\colon A\oplus_CB\to B\) are \(\Gamma\)-equivariant, they induce maps \(\widetilde\pr_1\colon (A\oplus_CB)\rtimes \Gamma\to A\rtimes\Gamma\) and \(\widetilde\pr_2\colon (A\oplus_CB)\rtimes \Gamma\to B\rtimes\Gamma\). By functoriality \(\widetilde\varphi\circ\widetilde\pr_1=\widetilde\psi\circ \widetilde\pr_2\), so that one obtains a \(^*\)-homomorphism
	\[(\widetilde\pr_1,\widetilde\pr_2)\colon (A\oplus_C B)\rtimes\Gamma\to (A\rtimes\Gamma)\oplus_{C\rtimes\Gamma}(B\rtimes\Gamma).\]
	To see that it is an isomorphism it suffices to show by \cite{Ped99}*{Prop.~3.1} the following three properties
	\begin{enumerate}
		\item\label{item:kernel-intersection} \(\ker(\widetilde\pr_1)\cap\ker(\widetilde\pr_2)=\{0\}\),
		\item\label{item:images} \(\widetilde\psi^{-1}(\widetilde\varphi(A\rtimes\Gamma))=\widetilde\pr_2((A\oplus_CB)\rtimes\Gamma)\),
		\item \label{item:image-kernel}\(\widetilde\pr_1(\ker \widetilde\pr_2)=\ker(\widetilde\varphi)\).
\end{enumerate}
	First, note that for a \(\Gamma\)-equivariant \(^*\)-homomorphism \(f\colon A\to B\) one has a short exact sequence
	\begin{equation*}
		\begin{tikzcd}
			0\arrow[r,""]&\ker(f)\arrow[r,""]&A\arrow[r,"f"] &f(A)\arrow[r,""]&0.
		\end{tikzcd}
	\end{equation*}
	As \(\ker f\) is a \(\Gamma\)-equivariant ideal, one has by \cref{res:ses-partial-actions-algebras} a short exact sequence
	\begin{equation*}
		\begin{tikzcd}
			0\arrow[r,""]&\ker(f)\rtimes\Gamma\arrow[r,""]&A\rtimes\Gamma\arrow[r,"\widetilde f"] &f(A)\rtimes\Gamma\arrow[r,""]&0.
		\end{tikzcd}
	\end{equation*}
	Hence, \(\ker(\tilde f)=\ker(f)\rtimes \Gamma\) and \(\tilde{f}(A\rtimes\Gamma)=f(A)\rtimes\Gamma\).
	In particular for \ref{item:kernel-intersection} this means \begin{equation*}\ker(\widetilde\pr_1)\cap\ker(\widetilde\pr_2)=\ker(\widetilde\pr_1)*\ker(\widetilde\pr_2) =(\ker(\pr_1)\rtimes \Gamma)*(\ker(\pr_2)\rtimes\Gamma).\end{equation*}
	For \(f_1\in L(\ker(\pr_1),\Gamma)\) and \(f_2\in L(\ker(\pr_2),\Gamma)\) one has \(f_1*f_2=0\) as \(\ker(\pr_1)\cdot\ker(\pr_2)=0\). By density \ref{item:kernel-intersection} follows. For \ref{item:images} one needs to show
		\begin{equation*}
		\widetilde\psi^{-1}(\varphi(A)\rtimes\Gamma)= \pr_2(A\oplus_CB)\rtimes\Gamma.
	\end{equation*}
	By density it suffices to show 
	\begin{equation*}
	\widetilde\psi^{-1}(L(\varphi(A),\Gamma))= L(\pr_2(A\oplus_CB),\Gamma).
	\end{equation*}
	 This is checked using \(\psi^{-1}(\varphi(A))=\pr_2(A\oplus_CB)\).
	Property \ref{item:image-kernel} is satisfied as
	\begin{equation*}
		\widetilde\pr_1(\ker \widetilde\pr_2)= \widetilde\pr_1(\ker(\pr_2)\rtimes\Gamma)=\pr_1(\ker \pr_2)\rtimes\Gamma = \ker(\varphi)\rtimes\Gamma = \ker(\widetilde\varphi).\qedhere
	\end{equation*}
\end{proof}
Using the description of partial crossed products as the \(C^*\)-algebra of the corresponding Fell bundle, see \cref{rem:fell-bundle}, the following statement follows from \cite{Exe17}*{Theorem~25.7}.
\begin{proposition}\label{res:partial-tensor-prod}
	Let \(\theta\) be a partial \(\Gamma\)-action on a \(C^*\)-algebra \(A\). For a further \(C^*\)-algebra \(B\), there is a partial \(\Gamma\)-action \(\theta\otimes\id\) on \(A\otimes_{\max}B\) given by \((A\otimes_{\max} B)_g=A_g\otimes_{\max} B\) and \((\theta\otimes\id)_g(a\otimes b)=\theta_g(a)\otimes b\) for \(a\in A_{g^{-1}}\) and \(b\in B\). Moreover, there is an isomorphism
	\begin{equation*}
		\Phi\colon(A\rtimes \Gamma)\otimes_{\max} B\cong (A\otimes_{\max}B)\rtimes \Gamma. 
	\end{equation*}
	with \(\Phi(f\otimes b)(g)=f(g)\otimes b\) for \(f\in L(A,\Gamma)\) compactly supported and \(b\in B\).
\end{proposition}
\begin{corollary}\label{res:partial-c0}
	Let \(A\) be a \(C^*\)-algebra with a partial \(\Gamma\)-action and equip \(C([0,1], A)\) with the partial \(\Gamma\)-action given by \(C([0,1], A)_g=C([0,1], A_g)\) and \((g\cdot f)(t)=g\cdot f(t)\) for \(g\in\Gamma\), \(f\in C([0,1], A)\) and \(0\leq t\leq 1\). 
	Then there is an isomorphism \(C([0,1],A\rtimes\Gamma)\to C([0,1],A)\rtimes\Gamma\) intertwining the evaluation map \(\ev_t\) and \(\widetilde\ev_t\) for all \(t\in[0,1]\).
\end{corollary}

\comment{to do: Check the maps between the spaces of equivalence classes of covariant representations are inverse to each other. }
%
%
In particular, one also gets \(C_0((0,1],A\rtimes \Gamma)\cong C_0((0,1],A)\rtimes\Gamma\). 
Applying \cref{res:partial_cossed-fibred-product} to the mapping cone, see \cref{ex:mapping-cone}, we obtain  the following result from  \cref{res:partial-c0}. 
\begin{corollary}\label{res:ses-mapping-cone-partial-action}
	Let \(A,B\) be \(C^*\)-algebras with partial \(\Gamma\)-actions and \(f\colon A\to B\) an equivariant \(^*\)-homomorphism and \(\tilde f\colon A\rtimes \Gamma\to B\rtimes \Gamma\) the induced map from \cref{res:induced-map-crossed-product}. Then
		there is a commuting diagram where the vertical rows are isomorphisms
	\begin{equation*}
		\begin{tikzcd}
			0\arrow[r,""]&C_f\rtimes\Gamma\arrow[r,""]\arrow[d,"\cong"]&Z_f\rtimes \Gamma\arrow[r,"\widetilde{\ev}_0"]\arrow[d,"\cong"] &B\rtimes \Gamma\arrow[r,""]\arrow[d,"="]&0\\
			0\arrow[r,""]&C_{\tilde f}\arrow[r,""]&Z_{\tilde f}\arrow[r,"\ev_0"] &B\rtimes \Gamma\arrow[r,""]&0.
		\end{tikzcd}
	\end{equation*}
\end{corollary}
\subsection{Globalization}
Given a partial action \(\theta\) of \(\Gamma\) on a locally compact Hausdorff space \(X\) one may ask if it is the restriction of a (global) group action of \(\Gamma\) on a larger space \(Y\) containing \(X\).
\begin{definition}
	Let \(\theta\) be a partial group action of \(\Gamma\) on a locally compact Hausdorff space \(X\). A topological space \(Y\) with a \(\Gamma\)-action \(\alpha\) is called a \emph{globalization} of the partial action, if 
	\begin{enumerate}
		\item \(X\subseteq Y\) is an open subset and \(X_g=X\cap\alpha_g(X)\) for all \(g\in\Gamma\),
		\item \(\theta_g=\alpha_g|_{X_{g^{-1}}}\) for all \(g\in\Gamma\),
		\item \(\bigcup_{g\in\Gamma}\alpha_g(X)=Y\).
	\end{enumerate}
\end{definition}
Such globalizations always exist and are unique, see \cite{Exe17}*{Proposition~5.5}. The proof is constructive. Namely, one sets \(Y=\Gamma\times X/\sim\) for the equivalence relation
\begin{equation*}
	(g,x)\sim(h,y) \Leftrightarrow x\in X_{g^{-1}h}\text{ and }\theta_{h^{-1}g}(x)=y.
\end{equation*}
Then \(X\) is embedded via \(x\mapsto[(e,x)]\) and \(\Gamma\) acts on \(Y\) by \(\alpha_g([h,x])=[gh,x]\). The space \(Y\) is not  Hausdorff in general; however,  this holds under the following condition.

\begin{proposition}[\cite{Exe17}*{Propositions~5.5, 5.7 and 28.4, Theorem~28.8}]\label{res:globalization}
	Let \(\theta\) be a partial action on a locally compact Hausdorff space \(X\) for which \(X_g\) are clopen in \(X\) for all \(g\in\Gamma\). Then it admits a globalization with a locally compact Hausdorff space \(Y\) and \(C_0(X)\rtimes\Gamma\) and \(C_0(Y)\rtimes \Gamma\) are Morita equivalent. 
\end{proposition}

\section{Partial actions for closed sets}\label{sec:closed}
Suppose now that \(W\) is a Riemannian manifold and that \(M\subset W\) is a compact submanifold of codimension~\(0\) with boundary \(\partial M\). We assume that \(\alpha\colon\Gamma\curvearrowright W\) is an action by smooth isometries. As \(M_g=M\cap \alpha_g(M)\) will be in general not an open subset of \(M\), the theory of partial group actions cannot be applied directly to \(M\). As in \cite{Sav18}, it is beneficial to construct a new manifold \(V\) by cutting \(W\) along the images of the boundary \(\partial M\) under the group action. Then the subspace \(Y\subseteq V\) corresponding to \(M\) will admit a partial \(\Gamma\)-action.
We strengthen the regularity assumption of \cite{BNSS22} as follows.
\begin{assumption}\label{ass:regularity1}
	Let \(Z_1,\ldots,Z_k\) denote the connected components of \(\partial M\). For all \(g,h\in\Gamma\) and \(Z_i,Z_j\) for \(i,j=1,\ldots,k\), one has either \(\alpha_g(Z_i)\cap \alpha_h(Z_j)=\emptyset\) or \(\alpha_g(Z_i) = \alpha_h(Z_j)\). Moreover, we assume that \(\Gamma\cdot \partial M=\bigcup_{g\in\Gamma}\alpha_g(\partial M)\) is a smooth closed submanifold of \(W\).
\end{assumption}

Under this assumption, $M$ will contain only finitely many images of  components of $\partial M$. Indeed, suppose $M$ contains infinitely many. Choosing a point in each of them we obtain a sequence with an accumulation point, say $z$, in $M$. Since $\Gamma\cdot \partial M$ is a closed manifold, so is its intersection with $M$. Hence $z \in \alpha_g(\partial M)$ for some $g\in \Gamma$, contradicting the fact that $\alpha_g (\partial M)$ is a manifold and therefore, near $z$, the zero set of a function with non-vanishing differential. 

\comment{Is there are better way to write these assumptions? For example, that \(\exists\, \varepsilon>0\) such that \(d(\alpha_g(Z_i),\alpha_h(Z_j))\geq \varepsilon\) whenever they are not equal. And deduce then that \(\Gamma\cdot\partial M\) is a closed submanifold. This should also imply that \(Y\) below has finitely many connected components?`}
\begin{example}\label{ex:shifts-closed}
Consider \(W=S^1\times\R\) with the shift action of \(\Z\) as in \cref{ex:shifts}, then taking \(M=S^1\times[0,\tfrac{3}{2}]\) satisfies the assumption, as \(\Gamma \cdot\partial M=S^1\times\tfrac{1}{2}\Z\).
\end{example}
\begin{example}
	Let \(W=\R^2\) be equipped with the scaling action of \(\Z\) as in \cref{ex:disk_skalings}. Then the closed disk \(M=\overline{\mathbb D}\) does not satisfy the assumptions as \(\Gamma\cdot\partial M=\bigcup_{k\in\Z}2^k\cdot S^1\) has \(0\) as a limit point. 
\end{example}
Under \cref{ass:regularity1} we can cut \(W\) along \(\Gamma\cdot\partial M\) to obtain a new manifold \(V\). 
\begin{definition}\label{def:cutted-manifold}
	Under \cref{ass:regularity1}, let \(V\) be the spherical blowup of \(\Gamma\cdot\partial M\) in \(W\). \comment{Add a good reference for spherical blowups?}
\end{definition}
As \(\Gamma\) acts by isometries and \(\Gamma\cdot\partial M\) is \(\Gamma\)-invariant, there is an induced \(\Gamma\)-action of \(V\) by functoriality of the spherical blowup. 
\begin{definition}\label{def:cutted-submanifold}
	Under \cref{ass:regularity1}, let \(Y\) denote the spherical blowup of \(\Gamma\cdot\partial M\cap M\) in \(M\).
\end{definition}
Then we can view \(Y\) as a submanifold of \(V\). \comment{Is this clear?} As \(Y\) is open in \(V\), it inherits a partial \(\Gamma\)-action \(\theta\) by \cref{res:action-restricts-to-partial-action}. In particular, this means that the \(\Gamma\)-action on \(V\) is a globalization of the partial \(\Gamma\)-action on \(Y\).
\begin{lemma}\label{res:partial-action-blowup}
	Restricting the induced \(\Gamma\)-action on \(V\) to \(Y\) yields a partial group action \(\theta\) of \(\Gamma\) on~\(Y\).
\end{lemma}
 As \(Y\) is also closed in \(V\), \(\alpha_g(Y)\) is closed for every \(g\in\Gamma\). This implies that \(Y_g= Y\cap \alpha_g(Y)\) is a clopen subset of \(Y\) for every \(g\in \Gamma\). Moreover, every \(Y_g\) is itself a smooth manifold with boundary. \comment{Is this clear as \(Y\) and \(\alpha_g(Y)\) have codimension \(0\) and the connected components of boundaries are either disjoint or equal?}
\begin{example}\label{ex:cylinder-continued}
	In \cref{ex:shifts-closed}, \(V\) is the disjoint union \(V=\bigsqcup_{k\in \Z}S^1\times[\frac{k}{2},\tfrac{k+1}{2}]\) and \(Y=S^1\times[0,\tfrac{1}{2}]\sqcup S^1\times[\tfrac{1}{2},1]\sqcup S^1\times[1,\tfrac{3}{2}]\). Here, we have \(Y_{-1}=S^1\times[0,\tfrac{1}{2}]\), \(Y_0=Y\), \(Y_1=S^1\times[1,\tfrac{3}{2}]\) and \(Y_k=\emptyset\) for all \(k\notin\{-1,0,1\}\).
\end{example}
\begin{lemma}\label{res:boundary-invariant}
	The interior \(Y^\circ=M\setminus(\Gamma\cdot\partial M)\) and the boundary \(\partial Y\) are \(\Gamma\)-invariant and \(\theta\) restricts to a partial action on \(Y^\circ\) and \(\partial Y\) with \((Y^\circ)_g=(Y_g)^\circ\) and \((\partial Y)_g=\partial(Y_g)\), respectively.
\end{lemma}
\begin{proof}
	It suffices to show that \(\partial Y\) is \(\Gamma\)-invariant, then its complement \(Y^\circ\) is also \(\Gamma\)-invariant.  But this is clear as \(\partial V\) is a \(\Gamma\)-invariant subset of \(V\).
\end{proof}
As for \(g\in\Gamma\), \(Y_g\) is clopen in \(Y\) and, similarly, \(\partial Y_g\) in \(\partial Y\), \cref{res:globalization} applies. As an example, we compute the \(K\)-theory of the crossed products for the \(\Z\)-action in the cylinder case.
\begin{example}
	Let \(V\) and \(Y\) be as in \cref{ex:cylinder-continued}.
	As the \(\Z\)-action on \(V\) globalizes the partial \(\Z\)-action on \(Y\), \(C(Y)\rtimes\Gamma\) is by \cref{res:globalization} Morita-equivalent to
	\begin{equation*}
		C(\Z\times S^1\times[0,\tfrac12])\rtimes\Z\cong \Compact \otimes C(S^1) \otimes C([0,1])
	\end{equation*}
	and hence, \(K_i(C(Y)\rtimes\Gamma)=\Z\) for \(i=1,2\). Similarly, \(C(\partial Y)\rtimes\Gamma\) is Morita-equivalent to 
	\begin{equation*}
		C(\Z\times S^1\times(\{0\}\sqcup\{\tfrac{1}{2}\})\rtimes\Z\cong \Compact \otimes C(S^1)\otimes (\C\oplus\C)
	\end{equation*}
	and thus, \(K_i(C(\partial Y)\rtimes\Gamma)=\Z^2\) for \(i=1,2\). As usual, \(\mathbb K\) denotes the compact operators. 
\end{example}

\section{Boutet de Monvel calculus}\label{sect:BdM}

\comment{Add references. Give a short overview of the calculus for \(Y\) Riemannian manifold. This should include:
	\begin{itemize}
		\item matrix structure, names and nature of the corresponding operator classes
		\item describe the order zero Boutet de Monvel algebra \(\Psi(Y,\partial Y)\subseteq \Bounded(L^2 Y\oplus L^2\partial Y)\) (in particular, which trace, Poisson and Green symbols of which degree are used)
		\item define \(\overline{\Psi(Y,\partial Y)}\) as the \(C^*\)-completion
\end{itemize}}

In the following we recall some properties of  Boutet de Monvel's algebra of operators of order and type zero. 

\subsection{Boutet de Monvel's calculus}\label{section:BoutetdeMonvelCalculus}
Let $Y$ be a compact manifold with boundary $\partial Y$, embedded in a  boundaryless manifold $\tilde Y$ of the same dimension, e.g., the double of $Y$. Moreover, let $E$ and $F$ be Hermitian  vector bundles over $Y$ and $\partial Y$, respectively. We fix a $\Gamma$-invariant Riemannian metric on $\tilde Y$ and endow $Y$ and $\partial Y$ with the measures induced by the Riemannian metric.

By $\Psi(Y,\partial Y)$ we denote the set of all operators of order and type zero in Boutet de Monvel's calculus, i.e.  all matrices of operators
\begin{eqnarray}\label{eq.D}
	D = \begin{pmatrix} P_++G&K\\T&S\end{pmatrix}: 
	\begin{array} {ccccc} 
		L^2(Y, E) & & L^2(Y, E) \\
		\oplus &\longrightarrow & \oplus\\
		L^2(\partial Y, F) & & L^2(\partial Y, F) 
	\end{array} .
\end{eqnarray}
Here, $P$ is a zero order classical pseudodifferential operator on $\tilde Y$, and $P_+$  acts on $u\in L^2(Y,E)$ by first extending $u$ by zero to a function $e^+u$ on $\tilde Y$, then restricting $Pe^+u$ to a function $r^+Pe^+u$ on  $Y$ so that $P_+ = r^+Pe^+$.  
For the calculus to work,  $P$ is required to satisfy the transmission condition at the boundary. This  is a condition on the homogeneous components of the symbol at the conormal in $S_{\partial Y}^*\tilde Y$.  Fix a local coordinate system $x=(x',x_n)$ near $\partial Y$ with $\partial Y $ locally given by $x_n=0$. Suppose the symbol $p$ of $P$ has the asymptotic expansion $p\sim\sum_{j=0}^\infty  p_{-j}$ with $p_{-j}$ homogeneous of degree $-j$ in $\xi$ for large $|\xi|$.  The transmission condition requires that, at $(\xi',\xi_n) = (0,\pm1)$, for all multi-indices $\alpha$, $\beta$, 
\begin{eqnarray*}
	D^\alpha_\xi D^\beta_x p_{-j}(x',0,0,-1) = (-1)^{-j-|\alpha|}   D^\alpha_\xi D^\beta_x p_{m-j}(x',0,0,1).
\end{eqnarray*}

The operator $S$ is a classical pseudodifferential operator of order zero on $\partial Y$;  $G$, $T$, and $K$  are  {\em singular Green}, {\em trace} and {\em potential} (or {\em Poisson}) operators, respectively, of order zero; $G$ and $T$ have type zero. In order to describe them we  introduce the subspaces $H^+$ and $H^-$ of $L^2(\R)$: 
$$H^\pm =\{\mathcal F(\chi_\pm v): v\in \mathcal S(\R)\}, $$ 
where $\chi_+$ is the characteristic function of $\R_+$,  $\chi_-$ that of $\R_-$, where $\R_\pm=\{t\in\R: t\gtrless 0\}$; $\mathcal F $ is the Fourier transform on $\R$.    

A trace symbol $t$ of order and type zero then is locally given by a function  $t=t(x',\xi',\xi_n) \in S^0_{\rm cl}(\R^{n-1} \times \R^{n-1}; H^-)$; a potential symbol of order zero is locally given by a function $k=k(x',\xi',\xi_n) \in S^{-1}_{\rm cl}(\R^{n-1} \times \R^{n-1}; H^+)$ and a singular Green symbol is locally given by a function $g=g(x',\xi',\xi_n,\eta_n)\in S^{-1}_{\rm cl}(\R^{n-1} \times \R^{n-1}; H^+\widehat\otimes_\pi H^-)$, see \cite[Definition 2.3.13]{Gru96} for more details.  
For fixed $(x',\xi')$ these symbols induce operators $k(x',\xi',D_n)\in \Bounded( \C ,L^2(\R_+))$, $t(x',\xi',D_n)\in \Bounded (L^2(\R_+) , \C)$ and  $g(x',\xi', D_n)\in \Bounded( L^2(\R_+))$, by  
\begin{eqnarray*}
	&&(k(x',\xi',D_n)c)(x_n)  = (2\pi)^{-\frac12}\int_0^\infty e^{ix_n\xi_n} k(x',\xi',\xi_n) \, d\xi_n c,\quad c\in \C,\\
	&& t(x',\xi',D_n) v = (2\pi)^{-\frac12} \int_0^\infty t(x',\xi',\xi_n) \mathcal F(e^+v)(\xi_n) \, d\xi_n, \quad v\in L^2(\R_+),\\
	&&(g(x',\xi',D_n) v)(x_n)  = (2\pi)^{-1} \int_0^\infty e^{ix_n\xi_n} g(x',\xi',\xi_n,\eta_n) \mathcal F(e^+v)(\eta_n) \, d\eta_n, \quad v\in L^2(\R_+),
\end{eqnarray*}
see \cite[(2.4.4), (2.4.5), (2.4.6)]{Gru96}. They are classical operator-valued symbols of order zero in the sense of \cite{S01}. The operators $G$, $T$ and $K$ are given by pseudodifferential quantization in $(x',\xi')$. This extends in a straightforward way to operators acting on sections of vector bundles. 

The operators in $\Psi(Y,\partial Y)$ form an  adjoint-invariant subalgebra of $\Bounded(L^2(Y,E)\oplus L^2(\partial Y,F))$. We denote by  
$\overline{\Psi(Y,\partial Y)}$ its closure in $\Bounded(L^2(Y,E)\oplus L^2(\partial Y,F))$, which is a $C^*$-algebra. 
For the purpose of this article it will be sufficient to consider the case where $E$ and $F$ are trivial one-dimensional complex bundles.

\subsection{Principal symbol maps}\label{sect:principal_symbol}

On $\overline{\Psi(Y,\partial Y)}$ two symbol maps are defined: the interior symbol 
\begin{eqnarray*}
	\sigma_{\interior}(D)=\sigma(D)|_{T^*Y} \in C^\infty(S^*Y) 
\end{eqnarray*}
and the boundary symbol $\sigma_{\boundary}(D)$. For the definition of the latter, we use the Riemannian metric to introduce $N\partial Y\subseteq T\tilde Y|_{\partial Y}$, the normal bundle to the boundary. \(N\partial Y\) can be trivialized by the inward-pointing normal unit vector field \(\nu\) and then \(N_z\partial Y_+\) denotes the half space \(\R_{+}\cdot\nu(z)\). 
Then $\sigma_{\boundary}(D)\in C^\infty(S^*\partial Y, \pi^*\End(L^2(N\partial Y_+)\oplus\C))$ is  given by
\begin{eqnarray*}\sigma_{\boundary} (D)(x',\xi') =
	\begin{pmatrix} 
		p^0(x',\xi',D_n) + g^{0}(x',\xi',D_n) &k^{0}(x',\xi',D_n) \\ t^{0}(x',\xi',D_n) & s^0(x',\xi') 
	\end{pmatrix} \in \Bounded(L^2(N_{x'}\partial Y_{+} )\oplus \C).
\end{eqnarray*}

with the operator-valued principal symbols 
of  \(p(x',\xi',D_n)\), \(g(x',\xi',D_n)\), \(k(x',\xi',D_n)\), \(t(x',\xi',D_n)\) and \(s(x',\xi')\).  
Here,  \(\End(L^2(N\partial Y_+)\oplus\C)\) denotes the vector bundle over \(\partial Y\) whose fiber over \(x'\in \partial Y\) is \(\Bounded(L^2(N_{x'}\partial Y_+)\oplus\C)\); $\pi^*$ denotes the pullback to $S^*\partial Y$.  

By continuity, both symbol maps have extensions, also denoted $\sigma_{\interior}$ and $\sigma_{\boundary}$, to  $\overline{\Psi(Y,\partial Y)}$:
\begin{equation}\label{eq:principal-symbol}
	(\sigma_{\interior},\sigma_{\boundary})\colon \overline{\Psi(Y,\partial Y)}\to C(S^*Y)\oplus C(S^*(\partial Y),\pi^*\End(L^2(N\partial Y_+)\oplus\C)).
\end{equation}

	\begin{remark}\label{rem:ident-with-trivial-bundle}
		One can identify the bundle \(\End(L^2(N\partial Y_+)\oplus\C)\) with the trivial bundle with fiber \(\Bounded(L^2(\R_+)\oplus \C)\) using the inward-pointing normal unit vector field \(\nu\).
	\end{remark}

Instead of considering the boundary symbol at \((x',\xi')\) as an operator on \(L^2(\R_+)\oplus \C\) one can also use the picture of Wiener--Hopf operators as described in \cite{BM71}*{Section~1}. 
Let $\Pi^\pm$ be the orthogonal projections in $L^2(\R)$ given by $u\mapsto \mathcal F\chi_\pm \mathcal F^{-1}$,  and  $\Pi_0^\pm: H^+\oplus H^-\to \C$ the map $u\mapsto \lim_{t\to0^\pm}(\mathcal F^{-1} u)(t)$.

Define \(\overline{H_+}=\mathcal F(L^2\R_+)\). In the following, it will be also useful to consider the following version of the boundary principal symbol.
\begin{definition}\label{def.hat_sigma_b}
	The boundary principal Wiener--Hopf operator of \(D\in\overline{\Psi(Y,\partial Y)}\) is given by
	\begin{equation*}
		\widehat {\sigma_{\boundary}}(D)(x',\xi')=(\mathcal F\oplus\id)\circ \sigma_{\boundary}(D)(x',\xi')\circ (\mathcal F^{-1}\oplus\id)\in\Bounded(\overline{H_+}\oplus\C)
	\end{equation*}
	for \((x',\xi')\in S^*\partial Y\).
	In more detail: 	For $(x',\xi')\in S^*\partial Y$,  
	\begin{eqnarray*}\label{eq:Fourier_boundary_symbol}
		\widehat{\sigma_{\boundary}} (D)(x',\xi') =
		\begin{pmatrix} 
			\Pi^+p_0(x',0,\xi',\xi_n) + \Pi^+_{0,\eta_n}g_{-1}(x',\xi',\xi_n,\eta_n) &k_{-1}(x',\xi',\xi_n) \\ \Pi^+_0 t_{0}(x',\xi',\xi_n) & s_0(x',\xi') 
		\end{pmatrix},
	\end{eqnarray*}
	considered as an element of $\Bounded(\Pi^+L^2(\R) \oplus \C)$. The operators act here as multiplication operators with respect to the $\xi_n$-variable. Note that 
	\begin{eqnarray*}
		(\Pi^+_{0,\eta_n}g_{-1}(x',\xi',\xi_n,\eta_n)u)(\xi_n) = \Pi^+_{0,\eta_n}(g_{-1}(x',\xi',\xi_n,\eta_n)u(\eta_n)),\quad u\in H^+.
	\end{eqnarray*}
\end{definition}

For more details on Boutet de Monvel's calculus see the original article \cite{BM71}, the monographs \cite{RS82}, \cite{Gru96}, or the short introduction \cite{S01}. 

\begin{remark}\label{rem:principal-symbol}
	One can show that \(\Compact(L^2(Y)\oplus L^2(\partial Y))= \ker(\sigma_{\interior},\sigma_{\boundary})\), see \cite[2.3.4 Theorem 1]{RS82}. Moreover, the induced map 
	\begin{equation*}(\overline{\sigma_{\interior}},\overline{\sigma_{\boundary}})\colon \Sigma=\overline{\Psi(Y,\partial Y)}/\Compact\to C(S^*Y)\oplus C(S^*(\partial Y),\pi^*\End(L^2(N\partial Y_+)\oplus\C))
	\end{equation*}
	is injective. Hence, by spectral invariance, an element of \(\overline{\Psi(Y,\partial Y)}\) is Fredholm if and only if its interior and boundary principal symbol are invertible. 
\end{remark}

\subsection{Toeplitz and Wiener-Hopf type operators}
Let \(f\in C^\infty(\overline\R)\), where \(\overline\R\) denotes the one point compactification of \(\R\). Then \(f\) defines a Toeplitz operator \(T_f\colon L^2(\R_+)\to L^2(\R_+)\) by
\begin{equation*}
	T_f=r^+\circ\mathcal F^{-1}\circ M_f\circ\mathcal F\circ e^+.
\end{equation*}
As before, \(r^+\colon L^2(\R)\to L^2(\R_+)\) and \(e^+\colon L^2(\R_+)\to L^2(\R)\) denote restriction and extension by~\(0\), \(\mathcal F\colon L^2(\R)\to L^2(\R)\) Fourier transform and \(M_f\) the multiplication operator associated with \(f\). Then ~\(f\) is called the symbol of \(T_f\). Note that \(T_f\) is unitarily equivalent to a usual Toeplitz operator via the Cayley transform. Denote by \(\mathfrak T\) the \(C^*\)-algebra generated by \(T_f\) with \(f\in C^\infty(\overline\R)\) and by \(\mathfrak T_0\) the ideal generated by \(T_f\) with \(f\in C^\infty(\overline\R)\) satisfying \(f(\infty)=0\). Note that \(\mathfrak T_0\) contains the compact operators on \(L^2(\R_+)\).

Consider the following subalgebras of \(\Bounded(L^2(\R_+)\oplus\C)\)
\begin{equation*}
	\mathfrak W=\begin{pmatrix}
		\mathfrak T & \ket{L^2(\R_+)}\\
		\bra{L^2(\R_+)} & \C
	\end{pmatrix},\quad\mathfrak W_0=\begin{pmatrix}
		\mathfrak T_0 & \ket{L^2(\R_+)}\\
		\bra{L^2(\R_+)} & \C
	\end{pmatrix}.
\end{equation*}
It is shown in \cite{MNS03}*{(8)} that \(\Sigma_{\boundary}:=\Image(\sigma_{\boundary})\) is contained in \(C_0(S^*\partial Y)\otimes \mathfrak W\). While the principal boundary symbol of the Green, Poisson, trace and pseudodifferential operator on the boundary give an element of \(C(S^*\partial Y)\otimes \Compact(L^2(\R_+)\oplus\C)\), a pseudodifferential operator \(P\) on \(Y\) with the transmission condition gives rise to a Toeplitz operator at \((x',\xi')\). Its symbol is \(p_0(x',0,\xi',\,\cdot\,)\) where \(p_0\) denotes the principal symbol of \(P\). Here, the transmission condition ensures that the symbol belongs to \(C^\infty(\overline \R)\).

Consider the map \(b\colon C(\partial Y)\to \Sigma_{\boundary}\) given by 
\begin{equation*}
	f\mapsto\begin{pmatrix*}
		f & 0\\
		0 & 0
	\end{pmatrix*}.
\end{equation*}
There is a split exact sequence by \cite{MNS03}*{Corollary 8}, see also \cite{MSS06}*{Theorem~4},
\begin{equation}\label{eq:ses-principal-boundary}
	\begin{tikzcd}
		0\arrow[r,""]&C(S^*\partial Y)\otimes\mathfrak W_0\arrow[r,""]&\Sigma_{\boundary}\arrow[r,""] &C(\partial Y)\arrow[r,""]&0,	
	\end{tikzcd}
\end{equation}	
where the split is given by \(b\).

\section{\(\Gamma\)-Boutet de Monvel algebra}\label{sec:partial-BM}
From now on, we will always assume the following situation.
\begin{assumption}
Let \(W\) be a Riemannian manifold and \(M\subseteq W\) a compact, codimension \(0\) submanifold with boundary \(\partial M\). Let \(\Gamma\) be a countable, discrete group which acts on \(W\) by smooth, orientation preserving isometries such that \cref{ass:regularity1} holds.
\end{assumption}
Then one can construct the space \(Y\) as in \cref{def:cutted-submanifold} and equip it with the partial \(\Gamma\)-action \(\theta\) as in \cref{res:partial-action-blowup}. Furthermore the boundary \(\partial Y\) is \(\Gamma\)-invariant, see \cref{res:boundary-invariant}.
The Riemannian metric on \(W\) induces volume forms \(\vol_Y\) on \(Y\) and \(\vol_{\partial Y}\) on \(\partial Y\) and the partial actions on \(Y\) and \(\partial Y\) give rise to partial isometries 
\begin{align}
	\nonumber
	U_g\colon L^2(Y,\vol_Y)&\to L^2(Y,\vol_Y)\\
	U_g(\varphi)(x)&=\begin{cases}
		\varphi(\theta_{g^{-1}}(x)) & \text{for }x\in Y_g,\\
	\nonumber
		0 &\text{else;}
	\end{cases}\\
	\nonumber
	V_g\colon L^2(\partial Y,\vol_{\partial Y})&\to L^2(\partial Y,\vol_{\partial Y})\\
	\label{eq:V_g2}
	V_g(\psi)(x)&=\begin{cases}
		 \psi(\theta_{g^{-1}}(x)) & \text{for }x\in \partial Y_g,\\
		0 &\text{else.}
	\end{cases}
\end{align}
\subsection{Partial \(\Gamma\)-action on the principal symbol algebra} 
We consider the induced partial \(\Gamma\)-action on the principal symbol algebra.
\paragraph{\textbf{Partial \(\Gamma\)-action on \(S^*Y\)}}
As \(\theta_g\colon Y_{g^{-1}}\to Y_g\) is a diffeomorphism, the differential \(d\theta_g\colon T(Y_{g^{-1}})\to T(Y_g)\) is an isomorphism. In the following denote by \(\partial\theta_g\colon T^*(Y_{g^{-1}})\to T^*(Y_{g})\) the codifferential \(\partial\theta_g=(d\theta^t)^{-1}\). Consider the partial \(\Gamma\)-action on \(T^*Y\) with \((T^*Y)_g=T^*(Y_g)\) and \(\tilde\theta_g\colon T^*Y_{g^{-1}}\to T^*Y_g\) given by
\begin{equation*}
	\tilde\theta_g(x,\xi)=\left(\theta_g(x),
	\partial_x\theta_g(\xi)\right).
\end{equation*}
Note that the partial action \(\tilde\theta\) restricts to a partial action on \(S^*Y\) with \((S^*Y)_g=S^*(Y_g)\). 

\paragraph{\textbf{Partial \(\Gamma\)-action on \(C(S^*(\partial Y),\End (L^2(N\partial Y_+)\oplus\C))\) }}
Similarly as above, the partial \(\Gamma\)-action on \(\partial Y\) induces a partial \(\Gamma\)-action on \(T^*(\partial Y)\) and \(S^*(\partial Y)\). Consider the normal bundle of \(\partial Y\) in \(Y\).
As \(d\theta_g\colon TY_{g^{-1}}\to TY_g\) restricts to an isomorphism \(T(\partial Y_{g^{-1}})\to T(\partial Y_g)\) and the Riemannian metric is \(\Gamma\)-invariant, \(d\theta\) restricts to partial \(\Gamma\)-action on \(N\partial Y\) with \((N\partial Y)_g=N(\partial Y_g)\). Note that the partial action maps the positive subspace \(N_{z}(\partial Y_{g^{-1}})_+\) to \(N_{\theta_g(z)}(\partial Y_g)_+\).
The Riemannian metric induces a measure \(\mu_z\) on every fiber \(N_z(\partial Y)\). 
Consider for \(g\in\Gamma\) and \(z\in \partial Y_g\)
\begin{align*}
	u_g\colon L^2(N_{\theta_g(z)}(\partial Y_{g^{-1}})_+,\mu_{\theta_g(z)})&\to L^2(N_z(\partial Y_g)_+,\mu_z)& \\
	u_g(\varphi)(X)&=
	\varphi(d\theta_{g^{-1}}(X)) & \text{for }X\in N_z(Y_g)_+.
\end{align*}
\begin{remark}\label{rem:ident-with-trivial-bundle-action}
	Under the identification \(\End(L^2(N\partial Y_+)\oplus\C)\) with the trivial bundle with fiber \(\Bounded(L^2(\R_+)\oplus \C)\), see \cref{rem:ident-with-trivial-bundle}, the unitaries \(u_g\) become the identity map \(L^2(\R_+)\to L^2(\R_+)\) as \(d\theta_{g^{-1}}\) is an orientation preserving linear isometry and must therefore map \(\nu\) to itself.
\end{remark}
Then one can define a partial \(\Gamma\)-action on \(C(S^*(\partial Y),\End (L^2(N\partial Y_+)\oplus\C))\) by
\begin{equation*}
	C(S^*(\partial Y),\End (L^2(N\partial Y_+)\oplus\C))_g = C(S^*(\partial Y_g),\End (L^2((N\partial Y_g)_+)\oplus\C))
\end{equation*}
and setting for \(f\in C(S^*(\partial Y_{g^{-1}}),\End (L^2(N\partial Y_{g^{-1}})\oplus\C))\) and \((z,\xi')\in S^*(\partial Y_g)\) 
\begin{equation}\label{eq:boundary-symbol-action}
	\tilde \theta^*_g(f)(z,\xi')= (u_g\oplus\id)\circ f(\tilde \theta_{g^{-1}}(z,\xi'))\circ (u_{g^{-1}}\oplus\id).
\end{equation}
It follows from \cite{RS82}*{p.~171}, see also \cite{BNSS22}*{Proposition 2}, that 
\begin{align}\label{eq:coordinate-changes}
	\sigma_{\interior}((U_g\oplus V_g)\circ T\circ(U_{g^{-1}}\oplus V_g^{-1})&=\tilde\theta_g^*\sigma_{\interior}(T)\\ \sigma_{\boundary}((U_g\oplus V_g)\circ T\circ(U_{g^{-1}}\oplus V_g^{-1})&=\tilde\theta^*_g\sigma_{\boundary}(T)\end{align}
for every \(T\in\overline{\Psi(Y,\partial Y)}\) and \(g\in\Gamma\).

Hence, the partial crossed products  $C(S^*Y)\rtimes\Gamma$ and $C(S^*(\partial Y),\End (L^2(N\partial Y_+)\oplus\C))\rtimes\Gamma$ are defined. Moreover, we can define the partial crossed product $\Sigma \rtimes\Gamma$ by setting
\begin{align*}
	\Sigma_g = \{[D]\in \Sigma\colon \overline{\sigma_{\interior}}([D])\in C(S^*Y_g) \text{ and }\overline{\sigma_{\boundary}}([D])\in C(S^*(\partial Y_g),\End (L^2(N\partial Y_+)\oplus\C))\}
\end{align*}
and \([D]\mapsto [(U_g\oplus V_g)\circ D \circ(U_{g^{-1}}\oplus V_g^{-1}] \), where \([D]\) denotes the equivalence class mod \(\Compact\).

\paragraph{\textbf{$\Gamma$-Boutet de Monvel operators}}
 
The algebraic partial crossed product $\Sigma\rtimes_{alg}\Gamma\subset \Sigma\rtimes\Gamma$ is the subset of  elements with finite supports in $\Gamma$. Define a homomorphism of algebras
 \begin{equation}
 \label{eq:quant3}
 \begin{array}{ccc}
    Q: \Sigma\rtimes_{alg}\Gamma & \longrightarrow&  \Bounded(L^2(Y)\oplus L^2(\partial Y))/\Compact \vspace{2mm} \\
    \{\sigma(D_g)\} &\longmapsto & \sum_{g\in\Gamma} D_g (U_g\oplus V_g) {\;\rm mod \;}\Compact.
 \end{array}
 \end{equation}
 Here, \(\sigma(D_g)\) is contained in \(\Sigma_g\).
 
\begin{definition}
	The \(C^*\)-algebra of \emph{$\Gamma$-Boutet de Monvel operators} associated with the partial \(\Gamma\)-action on \(Y\) is 
	\begin{equation}\label{eq:rep-of-boutet-de-monvel-crossed-prod}
		\nlBM=p^{-1}\left(\overline{Q(\Sigma\rtimes_{alg}\Gamma)}\right) \subset \Bounded(L^2(Y)\oplus L^2(\partial Y)) ,
	\end{equation}
	where $p:\Bounded\to \Bounded/\Compact$ is the natural projection to the Calkin algebra.
\end{definition}
\begin{remark}
	Note that one can identify \(L^2(Y,\vol_Y)\cong L^2(M,\vol_M)\), whereas \(\partial Y\) is a disjoint union of images of \(\partial M\) under the \(\Gamma\)-action. 
\end{remark}

\subsection{Fredholm criterion}
Next, we would like to find criteria when \(A\in\nlBM\) is Fredholm. 
We  make the following assumption.
\begin{assumption}\label{ass:topologically-free}
	The group \(\Gamma\) is amenable and the induced partial \(\Gamma\)-action on \(\Prim(\Sigma)\), for \(\Sigma=\overline{\Psi(Y,\partial Y)}/\Compact\), is topologically free. 
\end{assumption}

\begin{theorem}
	Suppose that \cref{ass:topologically-free} holds. Then the map~\eqref{eq:quant3} extends by continuity to the injective $*$-homomorphism $\overline{Q}:\Sigma\rtimes \Gamma\to \Bounded/\Compact$. In particular, \(A\in\nlBM\) is Fredholm if and only if the interior principal symbol \(\widetilde{ \sigma} _{\interior}(A)\in C(S^*Y)\rtimes\Gamma\) and the boundary principal  symbol \(\widetilde{\sigma}_{\boundary}(A)\in C(S^*(\partial Y),\End (L^2(N\partial Y_+)\oplus\C)) \rtimes\Gamma\) 	are invertible. 
\end{theorem}

\begin{proof}

1. Note that $Q$ is associated with a covariant homomorphism (cf.~\cite{Wil07}
and the definition of covariant representation in Section~\ref{sect:GroupoidAlgebra}) of the triple $(\Sigma,\Gamma,\tilde\theta^*)$ in the Calkin algebra $\Bounded/\Compact$. This means that we have an injective  unital $*$-homomorphism $\pi: \Sigma \to \Bounded/\Compact$ (natural injection) and a map $u:\Gamma\to \Bounded/\Compact$, $g\mapsto u_g=(U_g\oplus V_g)\mod\Compact$ and the following properties hold  for all \(g,h\in\Gamma\):
	\begin{enumerate}
		\item \(u_g\) is a partial isometry :  \( u_g=u_gu_g^* u_g\);
		\item \label{item:covariant33}\(u_g\pi(a)u_{g^{-1}}=\pi(\tilde\theta^*_g(a))\) for all \(a\in \Sigma_{g^{-1}}\);
		\item \(\pi(a)(u_gu_h-u_{gh})=0\) for all \(a\in \Sigma_g\cap \Sigma_{gh}\);
		\item \(u_g^*=u_{g^{-1}}\).
	\end{enumerate} 
	
2.  The Gelfand--Naimark--Segal theorem gives  a nondegenerate faithful representation of the Calkin algebra on a Hilbert space $\mathcal{H}$. Denote this representation by $i: \Bounded/\Compact \to \mathcal{B}\mathcal{H}$.  Then the pair $(i\circ\pi,i\circ u)$ is a covariant representation of $(\Sigma,\Gamma,\tilde\theta^*)$   in $\mathcal{H}$. Hence, by the universal property of the crossed product we have $*$-homomorphism:
$$
(i\circ \pi)\times(i\circ u) : \Sigma\rtimes \Gamma\longrightarrow \mathcal{B}\mathcal{H}.
$$
Moreover, the following diagram commutes:
\begin{equation*}
		\begin{tikzcd}
			  \Sigma\rtimes_{alg}\Gamma\arrow[r,""]\arrow[d,"Q"]&  \Sigma\rtimes \Gamma \arrow[d,"(i\circ \pi)\times(i\circ u)"] \\
			  \Bounded/\Compact\arrow[r,"i"] & \mathcal{B}\mathcal{H} .
		\end{tikzcd}
	\end{equation*}
Since $\Sigma\rtimes_{alg}\Gamma\subset  \Sigma\rtimes \Gamma $ is a dense subalgebra, we have
$$
i \left(\overline{Q(\Sigma\rtimes_{alg}\Gamma)}\right) =(i\circ \pi)\times(i\circ u) (\Sigma\rtimes \Gamma).
$$
This implies that $Q$ extends by continuity to the map of the crossed product  
\begin{equation}
\label{eq:mainmap}
 \overline{Q}:\Sigma\rtimes \Gamma\to \Bounded/\Compact,\qquad Q(w)=i^{-1}\left((i\circ \pi)\times(i\circ u)w\right).
\end{equation}

3. Recall from \cref{rem:principal-symbol} that the interior and boundary principal symbol induce an injective \(^*\)-homomorphism
\(Q:\Sigma\to \Bounded/\Compact\).  Hence, $ i \circ Q:\Sigma \to \mathcal{B}\mathcal{H}$ is also injective. 
Under  \cref{ass:topologically-free},  \cref{res:isomorphism-theorem} implies that the representation \((i\circ \pi)\times(i\circ u)\) of \(\Sigma\rtimes\Gamma\) is faithful. 

Hence, \eqref{eq:mainmap} is injective.  This implies the desired Fredholm criterion by the Atkinson-Nikol'skii theorem.
\end{proof}

\begin{example}
	Suppose the original partial action \(\Gamma\curvearrowright M\) is free. Then also the induced partial actions on \(Y\) and \(S^*Y\) and the \(\Gamma\)-invariant subsets \(\partial Y\) and \(S^*\partial Y\) are free. 	
	In \cite{MSS06} the following composition series is considered
	\begin{equation}\label{eq:composition-series}
		0\subseteq \Compact\subseteq \ker(\sigma_{\interior})\subseteq \overline{\Psi(Y,\partial Y)}.
	\end{equation}
	It is shown in \cite{MSS06}*{Theorem~6} that the boundary principal symbol induces an isomorphism
	\begin{equation*}
		\ker(\sigma_{\interior})/\Compact \cong C(S^*\partial Y)\otimes\Compact(L^2(\R_+\oplus\C))
	\end{equation*}
	and in \cite{MSS06}*{Corollary~26} that the interior principal symbol induces an isomorphism
	\begin{equation*}
		\overline{\Psi(Y,\partial Y)}/\ker(\sigma_{\interior}) \cong C(S^*Y/\sim).
	\end{equation*}
	Here, \(\sim\) denotes the equivalence relation with \((x',0,0,1)\sim(x',0,0,-1)\) for \(x'\in \partial Y\) which is obtained due to the transmission condition.	Hence, we see that as sets
	\begin{equation*}
		\Prim(\Sigma)=\Prim(\ker(\sigma_{\interior})/\Compact)\cup \Prim(\overline{\Psi(Y,\partial Y)}/\ker(\sigma_{\interior}))=S^*\partial Y\cup (S^*Y/\sim).
	\end{equation*}
	Note that the composition series consists of \(\Gamma\)-invariant ideals and that \(\sim\) is respected by the induced partial \(\Gamma\)-action on \(S^*Y\). Hence, we see that the \(\Gamma\)-action on \(\Prim\Sigma\) is free and, in particular, \cref{ass:topologically-free} holds.
\end{example}

\subsection{Operators of arbitary order on \(\partial Y\)}
As a byproduct of the construction above, we get in the lower right corner a calculus of operators on \(L^2(\partial Y)\) we denote by \(\overline{\Psi^0_\Gamma(\partial Y)}\). It has a principal symbol map
\begin{align*}
	\sigma_0 \colon \overline{\Psi^0_\Gamma(\partial Y)} \to C(S^*\partial Y)\rtimes\Gamma.
\end{align*}
The \(C^*\)-algebra \(\overline{\Psi_\Gamma^0(\partial Y)}\) is generated by sums
\(\sum_{g\in\Gamma} A_g V_g \) with \(V_g\) defined in \eqref{eq:V_g2} and 
\begin{align*}
	A_g\in \overline{\Psi^0(\partial Y)}_g = \{A\in \overline{ \Psi^0(\partial Y)} \colon \sigma_0(A)|_{S^*\partial Y\setminus S^*\partial Y_g}=0\}.
\end{align*}
Note that these sums are always finite, as there are only finitely many \(\partial Y_g\neq \emptyset\).

Similary, as in \cite{NSS08}*{Section~2.1.3} for honest group actions we can define a calculus of arbitary order on \(\partial Y\) as follows.
\begin{definition}\label{def:smooth-crossed-product}
	Let \(X\) be a smooth compact manifold with a partial action of \(\Gamma\) by diffeomorphisms \(\theta_g\colon X_{g^{-1}}\to X_g\) and assume that there are only finitely many \(X_g\neq \emptyset\). Then denote by \(C^\infty(X)\rtimes \Gamma\)  the subspace of \(C(X)\rtimes \Gamma\subseteq \Bounded(L^2(X))\) of
	\begin{align*}
		\sum_{g\in \Gamma} \pi(f_g) u_g
	\end{align*}
	where \(f_g\in C^\infty(X_g)\), \(\pi\) denotes the representation of \(C(X)\) on \(L^2(X)\) by multiplication and \(u_g\) the shifts induced by the action.
\end{definition} 
It is easy to check that \(C^\infty(X)\rtimes \Gamma\) is a subalgebra of \(C(X)\rtimes\Gamma\).

\begin{example}
	Consider the induced partial action on \(S^*\partial Y\).
	The maps \(d\theta_g^*\colon C(S^*\partial Y_g)\to C(S^*\partial Y_{g^{-1}})\) restrict to maps \(d\theta_g^*\colon C^\infty(S^*\partial Y_g)\to C^\infty(S^*\partial Y_{g^{-1}})\). Therefore, we can define the smooth partial crossed product \(C^\infty(S^*\partial Y)\rtimes \Gamma\).
\end{example}

Note that also that \(V_g\colon L^2(\partial Y)\to L^2(\partial Y)\) can be restricted to a map \(V_g\colon C^\infty(\partial Y)\to C^\infty(\partial Y)\).
\begin{definition}
	Let \(\Psi^m_{\Gamma}(\partial Y)\) denote the subspace of operators on \(C^\infty(\partial Y)\) of the form
	\(
	\sum_{g\in \Gamma} A_g V_g
	\)
	with \(A_g\) belonging to 
	\begin{align*}
		\Psi^m(\partial Y)_g = \{A\in  \Psi^m(\partial Y) \colon \sigma_m(A)|_{S^*\partial Y\setminus S^*\partial Y_g}=0\}.
	\end{align*}
	Define a principal symbol map
	\begin{align*}\tilde\sigma_m\colon \Psi^m_{\Gamma}(\partial Y)&\to C^\infty(S^*\partial Y)\rtimes \Gamma \\
		\sum_{g\in \Gamma} A_g V_g &\mapsto \pi(\sigma_m(A_g))u_g.
	\end{align*}
\end{definition} 
Let \(H^s(\partial Y)\) denote the Sobolev space defined with respect to the Laplace--Beltrami operator \(\Delta_g\) of the Riemannian metric. As the metric is \(\Gamma\)-invariant, the operators \(V_g\) extend to partial isometries on all Sobolev spaces. Hence, \(A\in \Psi^m_\Gamma(\partial Y)\) extends to a continuous operator \(H^s(\partial Y)\to H^{s-m}(\partial Y)\) for every \(s\in\R\). 
Similary to \cite{NSS08}*{Theorem~2.3} the following result holds.
\begin{proposition}\label{res:smooth-calculus}
	The calculus on \(\partial Y\) has the following properties:
	\begin{enumerate}
		\item\label{item:well_defined} The principal symbol map \(\tilde\sigma_m\) is well-defined, i.e. does not depend on the representation of the operator as a sum \(\sum_{g\in \Gamma} A_g V_g\).
		\item\label{item:kernel} If the operator \(A\in \Psi^m_\Gamma(\partial Y)\) satisfies \(\tilde\sigma_m(A)=0\), then \(A\colon H^s(\partial Y)\to H^{s-m}(\partial Y)\) is compact for every \(s\in \R\).
		\item\label{item:product} For \(A\in \Psi_\Gamma^m(\partial Y)\) and \(B\in \Psi_\Gamma^{m'}(\partial Y)\) one has \(AB\in \Psi_\Gamma^{m+m'}(\partial Y)\) and \(\tilde\sigma_{m+m'}(AB)=\tilde\sigma_m(A)\tilde\sigma_{m'}(B)\).
	\end{enumerate}
\end{proposition}
\begin{proof}
	Let \(\Lambda=(I-\Delta_g)^{1/2}\), then \(V_g\) commutes with all powers of \(\Lambda\) as the metric was assumend to be \(\Gamma\)-invariant. Hence, for every operator \(A\in \Psi^m_\Gamma(\partial Y)\) one has \(\Lambda^{-m}A\in \Psi^0_\Gamma(\partial Y)\) and \(A\Lambda^{-m}\in \Psi^0_\Gamma(\partial Y)\). Moreover, the principal symbol stays the same under this order reduction, that is, \(\tilde\sigma_{0}(\Lambda^{-m}A)=\tilde\sigma_{0}(A\Lambda^{-m})=\tilde\sigma_m(A)\). This can be used to show the result using that the claimed properties hold true for \(m=s=m'=0\) as shown previously.
\end{proof}

\section{\(K\)-theory of the symbol algebra}\label{sect:Ktheory}
In the following, we show, analogously to the results of \cite{MSS06}, that the \(K\)-theory of the symbol algebra \(\Sigma\rtimes\Gamma\) splits.

Let \(\Sigma_{\boundary}=\Image(\sigma_{\boundary})\subset C(S^*\partial Y, \End(L^2(N^*\partial Y)\oplus\C)\). Then there is a \(\Gamma\)-equivariant commuting diagram
\begin{equation*}
	\begin{tikzcd}
		0\arrow[r,""]&\ker(\sigma_{\boundary})/\Compact\arrow[r,"\iota"]&\Sigma\arrow[r,"\overline{\sigma_{\boundary}}"] &\Sigma_{\boundary}\arrow[r,""]&0\\
		0\arrow[r,""]&C_0(Y^\circ)\arrow[r,""]\arrow[u,"m_1"]&C(Y)\arrow[r,"{f\mapsto f|_{\partial Y}}",swap]\arrow[u,"m"] &C(\partial Y)\arrow[r,""]\arrow[u,"b"]&0
	\end{tikzcd}
\end{equation*}
with \(m_1,m,b\) defined by
\begin{equation}\label{eq:m_0,m,b}
	m_1(f)=\left[\begin{pmatrix*}
		f & 0\\
		0 & 0
	\end{pmatrix*}\right], 
	\quad m(f)=\left[\begin{pmatrix*}
		f & 0\\
		0 & 0
	\end{pmatrix*}\right], 
	\quad b(g)=\begin{pmatrix*}
		g & 0\\
		0 & 0
	\end{pmatrix*}.
\end{equation}
Similarly to \cite{MSS06}*{(10)} there is a corresponding diagram of the partial crossed products of the associated mapping cone sequences with exact rows and columns. 
\begin{equation}\label{eq:ses-mapping-cones}
	\begin{tikzcd}
		&0&0 &0&\\
		0\arrow[r,""]&\ker(\sigma_{\boundary})/\Compact\rtimes\Gamma\arrow[r,""]\arrow[u]&\Sigma\rtimes\Gamma\arrow[r,"\widetilde{\overline{\sigma_{\boundary}}}",swap]\arrow[u] &\Sigma_{\boundary}\rtimes\Gamma\arrow[r,""]\arrow[u]&0\\
		0\arrow[r,""]&Z_{m_1}\rtimes\Gamma\arrow[r,""]\arrow[u,"\widetilde\ev_0"]&Z_{m}\rtimes\Gamma\arrow[r]\arrow[u,"\widetilde\ev_0"] &Z_b\rtimes\Gamma\arrow[r,""]\arrow[u,"\widetilde\ev_0"]&0\\
		0\arrow[r,""]&C_{m_1}\rtimes\Gamma\arrow[r,""]\arrow[u]&C_{m}\rtimes\Gamma\arrow[r]\arrow[u] &C_b\rtimes\Gamma\arrow[r,""]\arrow[u]&0\\
		&0\arrow[u]&0\arrow[u] &0\arrow[u]&.\\
	\end{tikzcd}
\end{equation}
\begin{lemma}\label{res:boundary-cone-trivial}
	One has \(K_i(C_b\rtimes\Gamma)=0\) for \(i=0,1\), and consequently the induced map \(K_i(C_{m_1}\rtimes\Gamma)\to K_i(C_m\rtimes\Gamma)\) is an isomorphism for \(i=0,1\).
\end{lemma}
\begin{proof}
	By the long exact sequence in \(K\)-theory associated with the mapping cone, it suffices to show that the map \(\tilde b_*\colon K_i(C(\partial Y)\rtimes\Gamma)\to K_i(\Sigma_{\boundary}\rtimes\Gamma)\) is an isomorphism for \(i=0,1\).
	
	Recall from \eqref{eq:ses-principal-boundary} that there is a split exact sequence 
	\begin{equation*}
		\begin{tikzcd}
			0\arrow[r,""]&C(S^*\partial Y)\otimes\mathfrak W_0\arrow[r,""]&\Sigma_{\boundary}\arrow[r,""] &C(\partial Y)\arrow[r,""]&0,	
		\end{tikzcd}
	\end{equation*}	
	where the split is given by \(b\). From \eqref{eq:boundary-symbol-action} and \cref{rem:ident-with-trivial-bundle-action}, we see that the \(\Gamma\)-action on the \(\Bounded(L^2(\R_+)\oplus\C)\)-part is actually trivial. Hence, using \cref{res:partial-tensor-prod}, there is a short exact sequence 
	\begin{equation*}
		\begin{tikzcd}
			0\arrow[r,""]&C(S^*\partial Y)\rtimes\Gamma\otimes\mathfrak W_0\arrow[r,""]&\Sigma_{\boundary}\rtimes\Gamma\arrow[r,""] &C(\partial Y)\rtimes\Gamma\arrow[r,""]&0	
		\end{tikzcd}
	\end{equation*}	
	with split \(\tilde b\). Hence, one needs to show that \(K_i(C(S^*(\partial Y))\rtimes\Gamma\otimes\mathfrak W_0)=0\) for \(i=0,1\). In the proof of \cite{MNS03}*{Lemma 7} it is shown that
	\begin{equation}
		\label{eq:frak_W-zero-k-theory}
		K_i(\mathfrak W_0)=0 \quad\text{for }i=0,1.
	\end{equation}
	Consequently, by the K\"unneth Theorem, see for example \cite{Bla98}*{Theorem~23.1.3}, 
	 one deduces \(K_i(C(S^*(\partial Y))\rtimes\Gamma\otimes\mathfrak W_0)=0\).
\end{proof}
Arguing analogously to \cite{MSS06}*{p.~228} we obtain the following splittings.
\begin{lemma}\label{res:split-1}
The long exact sequence in \(K\)-theory of the first column in \eqref{eq:ses-mapping-cones} splits and for \(i=0,1\)
	\begin{equation*}
		K_i(\ker(\sigma_{\boundary})/\Compact\rtimes\Gamma)\cong K_{1-i}(C_{m_1}\rtimes\Gamma)\oplus K_i(C_0(Y^\circ)\rtimes \Gamma).
	\end{equation*}
\end{lemma}
\begin{proof}
	The interior principal symbol map induces a \(\Gamma\)-equivariant isomorphism \(\ker(\sigma_{\boundary})/\Compact\to C_0(S^*Y^{\circ})\) under which the map \(m_1\) corresponds to \(\pi^*\), where \(\pi\colon S^*Y^\circ\to Y^\circ\) is the bundle projection. Under this isomorphism the \(K\)-theory long exact sequence of the first row is given by
	\begin{equation*}
		\begin{tikzcd}
			\cdots\arrow[r,""]&K_i(C(Y^\circ)\rtimes\Gamma)\arrow[r,"\widetilde{\pi^*}"]&K_i(C_0(S^*Y^\circ)\rtimes\Gamma)\arrow[r,""] &K_{i+1}(C_{\pi^*}\rtimes\Gamma)\arrow[r,""]&\cdots	.
		\end{tikzcd}
	\end{equation*}	
	Choosing a non-vanishing section of \(T^*Y^\circ\) yields a split \(s\) of \(\pi^*\colon C(Y^\circ)\to C(S^*Y^\circ)\). Note that the existence of a non-vanishing split can be shown as in \cite{MNS03}*{Proposition~9} although \(Y\) is not connected, since \(Y\) has finitely many connected components which all have a nonempty boundary, see the remark after \cref{ass:regularity1}. Then \(\widetilde s\) is a split of \(\widetilde{ \pi^*}\), which induces the claimed split of the \(K\)-theory sequence.
\end{proof}
\begin{corollary}\label{res:split-2}
	\comment{see comment in the previous lemma.} The long exact sequence in \(K\)-theory of the second column in \eqref{eq:ses-mapping-cones} splits and for \(i=0,1\)
	\begin{equation*}
		K_i(\Sigma\rtimes\Gamma)\cong K_{1-i}(C_{m}\rtimes\Gamma)\oplus K_i(C(Y)\rtimes \Gamma).
	\end{equation*}
\end{corollary}
\begin{proof}
	Consider the following commuting diagram induced from \eqref{eq:ses-mapping-cones}
	 \begin{equation*}
	 	\begin{tikzcd}
	 		\cdots\arrow[r]&K_i(C(Y)\rtimes\Gamma)\arrow[r]&K_i(\Sigma\rtimes\Gamma)\arrow[r,""]&K_{1-i}(C_m\rtimes\Gamma)\arrow[r]&\cdots\\
	 		\cdots\arrow[r]&K_i(C( Y^\circ)\rtimes\Gamma)\arrow[r]\arrow[u]&K_i(\ker(\sigma_{\boundary})/\Compact\rtimes\Gamma)\arrow[r,""]\arrow[u,"\tilde\iota_*"]&K_{1-i}(C_{m_1}\rtimes\Gamma)\arrow[u,"\cong"]\arrow[r]&\cdots.
	 	\end{tikzcd}
	 \end{equation*}
	 As the right vertical map is an isomorphism by \cref{res:boundary-cone-trivial} and the lower sequence splits, one can use a split of \(K_i(\ker(\sigma_{\boundary})/\Compact\rtimes\Gamma)\to K_{1-i}(C_{m_1}\rtimes\Gamma)\) to construct a split for the upper sequence. 
\end{proof}
\section{Homotopy classification}\label{sect:homotopy}
In the following, we will recall the definition of \(\Ell\)-groups as in \cite{Sav05} and prove a homotopy classification of nonlocal Boutet de Monvel operators. 

\subsection{Ell-groups}
For a \(C^*\)-subalgebra \(\mathcal A\subseteq \Bounded(\Hils)\) containing \(\Compact(\Hils)\) together with a \(C^*\)-algebra \(\mathcal A_0\subseteq \mathcal A\), \cite{Sav05} introduces a group \(\Ell(\mathcal A_0,\mathcal A)\) generalizing the classical \(\Ell\)-group for \(\mathcal A\) being the \(C^*\)-algebra of order zero pseudodifferential operators \(\overline{\Psi(M)}\) and \(\mathcal A_0=C(M)\). 
\begin{definition}	
	A \emph{compatible triple} \(\mathbf D =(D,P,Q)\) consists of two projections \(P,Q\in\Mat_n(\mathcal A_0)\) and \(D\in\Mat_n(\mathcal A)\) such that \(D=QDP\). We then consider the operator \(D|_{\Image(P)}\colon\Image(P)\to \Image(Q)\), the so-called  \emph{operator in subspaces determined by the projections \(P\) and \(Q\)}.
	
	Let \(\sigma\colon \mathcal A\to\mathcal A/\Compact(\Hils)\) denote the quotient map. Then \(\mathbf D=(D,P,Q)\) is called \emph{elliptic} if there is some \(D'\in \Mat_n(\mathcal A)\) such that
	\begin{equation}\label{eq:parametrix}
		\sigma(D)\sigma(D')=\sigma(Q) \quad\text{ and }\quad\sigma(D')\sigma(D)=\sigma(P).
	\end{equation}
\end{definition}
Note that for an elliptic \(\mathbf D=(D,P,Q)\), the operator  \(D|_{\Image(P)}\colon\Image(P)\to \Image(Q)\) is Fredholm. 

The direct sum of two elliptic compatible triples \(\mathbf D_0=(D_0,P_{0},Q_{0})\) and \(\mathbf D_1=(D_1,P_1,Q_1)\)
given by \(\mathbf D_0\oplus \mathbf D_1=(D_0\oplus D_1,P_0\oplus P_1,Q_0\oplus Q_1)\) is again an elliptic compatible triple.
\begin{definition}
	Two elliptic compatible triples \(\mathbf D_0=(D_0,P_{0},Q_{0})\) and \(\mathbf D_1=(D_1,P_1,Q_1)\) are \emph{homotopic} if there are continuous maps \([0,1]\to\Mat_n(\mathcal A)\), \(t\mapsto D_t\) and \([0,1]\to\Mat_n(\mathcal A_0)\), \(t\mapsto P_{t}\) and \(t\mapsto Q_t\), such that \((D_t,P_t,Q_t)\) is an elliptic compatible triple for every \(t\in[0,1]\).
	
	An elliptic compatible triple \(\mathbf D =(D,P,Q)\) is \emph{trivial} if \(D\in\Mat_n(\mathcal A_0)\). Two elliptic compatible triples \(\mathbf D_0=(D_0,P_0,Q_0)\) and \(\mathbf D_1=(D_1,P_1,Q_1)\) are \emph{stably homotopic} if there are some trivial elliptic \(\mathbf D',\mathbf D''\) such that \(\mathbf D_1\oplus \mathbf D'\) and \(\mathbf D_2\oplus \mathbf D''\) are homotopic.
	
	Moreover, \(\Ell(\mathcal A_0,\mathcal A)\) is defined as stable homotopy classes of elliptic compatible triples. 
\end{definition}
Then \(\Ell(\mathcal A_0,\mathcal A)\) is a group with respect to the direct sum and \([(D,P_1,P_2)]^{-1}=[(D',P_2,P_1)]\) where \(D'\) denotes a parametrix as in \eqref{eq:parametrix}.
The following result is shown in \cite{Sav05}.
\begin{theorem}[\cite{Sav05}*{Theorem 4}]
The group \(\Ell(\mathcal A_0, \mathcal A)\) is isomorphic to \(K_0(C_{f})\). Here, \(f\colon\mathcal A_0\to \mathcal A/\Compact(\Hils)\) denotes the composition of the inclusion \( \mathcal A_0\hookrightarrow \mathcal A\) with the quotient map \(\sigma\colon \mathcal A\to \mathcal A/\Compact(\Hils)\) and \(C_f\) is the mapping cone see \cref{ex:mapping-cone}.
\end{theorem}

One way to define this isomorphism is to identify \(K_0(C_f)\) with the relative \(K\)-group \(K_0(f)\), see for example \cite{Has24}, see also the proof of \cite{CMR07}*{Thm.~2.38}. For a unital \(^*\)-homomorphism \(f\colon A\to B\) a relative \(K\)-cycle is  a triple \((p,q,x)\) with projections \(p,q\in\Mat_\infty(A)\) and \(x\in\Mat_\infty(B)\) for which there is \(y\in\Mat_\infty(B)\) such that \(xy=f(q)\) and \(yx=f(p)\). The corresponding \(K\)-group is defined by taking equivalence classes with respect to stable homotopy. 
Then the isomorphism \(\Ell(\mathcal A_0,\mathcal A)\to K_0(C_f)\) is given by \([(D,P,Q)]\mapsto[(P,Q,\sigma(D)]\).
\subsection{Boutet de Monvel case}
We consider the \(C^*\)-algebra of nonlocal Boutet de Monvel operators \(\nlBM\) defined in \eqref{eq:rep-of-boutet-de-monvel-crossed-prod}. Let \(\mathcal A_0\) be the subalgebra \((\pi\times(U\oplus V))(C(Y)\rtimes \Gamma\oplus C(\partial Y)\rtimes \Gamma)\).

\begin{lemma}
	Suppose that \cref{ass:topologically-free} holds and the partial action of \(\Gamma\) on \(Y\) is topologically free. Then one has
	\(
	\Ell(\mathcal A_0,\nlBM)\cong K_0(C_{\tilde m_2})
	\)
	where \(\tilde m_2\colon C(Y)\rtimes\Gamma\oplus C(\partial Y)\rtimes\Gamma\to\Sigma\rtimes\Gamma\) is the map induced by 
	\begin{equation*}
		m_2(g,h)=\left[\begin{pmatrix*}
			g& 0\\
			0 &h
		\end{pmatrix*}\right].
	\end{equation*}
\end{lemma}
\begin{proof}
	Consider the commuting diagram
	\begin{equation*}
		\begin{tikzcd}
			C(Y)\rtimes\Gamma\oplus C(\partial Y)\rtimes\Gamma \arrow[d,"\pi\times(U\oplus V)"]\arrow[r,"\tilde m_2"]& \Sigma\rtimes\Gamma\arrow[d,"\rho\times\tilde U"]\\
			\mathcal A_0 \arrow[r,"f"] &\nlBM/\Compact.
		\end{tikzcd}
	\end{equation*}
\cref{res:isomorphism-theorem} in connection with 	\cref{ass:topologically-free} implies that the right map is an isomorphism. 
	To see that the left map is also an isomorphism by \cref{res:isomorphism-theorem}, we claim that \cref{ass:topologically-free} ensures that the partial action of \(\Gamma\) on \(\partial Y\) is topologically free. Namely, one can consider the composition series \eqref{eq:composition-series}. Then \(S^*\partial Y=\Prim(\ker(\sigma_{\interior})/\Compact)\) is an open, \(\Gamma\)-invariant subset of \(\Prim(\Sigma)\). This means that the partial \(\Gamma\)-action on \(S^*\partial Y\) must also be topologically free and likewise on \(\partial Y\).
	Hence, both vertical maps are isomorphisms so that \(
	\Ell(\mathcal A_0,\nlBM)\cong K_0(C_f)\cong K_0(C_{\tilde m_2})
	\).
\end{proof}
Recall that \(\Sigma_{\boundary}=\Image(\sigma_{\boundary})\subset C(S^*\partial Y, \End(L^2(N^*\partial Y)\oplus\C)\) and the definition of the map \(m_1\) from \eqref{eq:m_0,m,b}. Then there is a \(\Gamma\)-equivariant commuting diagram
\begin{equation*}
	\begin{tikzcd}
		0\arrow[r,""]&\ker(\sigma_{\boundary})/\Compact\arrow[r,""]&\Sigma\arrow[r,"\overline{\sigma_{\boundary}}"] &\Sigma_{\boundary}\arrow[r,""]&0\\
		0\arrow[r,""]&C_0(Y^\circ)\arrow[r,""]\arrow[u,"m_1"]&C(Y)\oplus C(\partial Y)\arrow[r,"{(g,h)\mapsto (g|_{\partial Y},h)}",swap]\arrow[u,"m_2"] &C(\partial Y)\oplus C(\partial Y)\arrow[r,""]\arrow[u,"m_3"]&0
	\end{tikzcd}
\end{equation*}
with \(m_3\) given by
\begin{equation*}
 m_3(h_1,h_2)=\begin{pmatrix*}
		h_1 & 0\\
		0 & h_2
	\end{pmatrix*}.
\end{equation*}
Consequently one obtains a commuting diagram
\begin{equation*}
	\begin{tikzcd}
		0\arrow[r,""]&\ker(\sigma_{\boundary})/\Compact\rtimes\Gamma\arrow[r,""]&\Sigma\rtimes\Gamma \arrow[r,"\widetilde{\overline{\sigma_{\boundary}}}"] &\Sigma_{\boundary}\rtimes \Gamma\arrow[r,""]&0\\
		0\arrow[r,""]&C_0(Y^\circ)\rtimes \Gamma \arrow[r,""]\arrow[u,"\tilde m_1"]&C(Y)\rtimes\Gamma\oplus C(\partial Y)\rtimes\Gamma\arrow[r,""]\arrow[u,"\tilde m_2"] &C(\partial Y)\rtimes\Gamma\oplus C(\partial Y)\rtimes\Gamma\arrow[r,""]\arrow[u,"\tilde m_3"]&0.
	\end{tikzcd}
\end{equation*}
By the compatibility of partial actions with the mapping cone construction, see \cref{res:ses-mapping-cone-partial-action}, this yields a short exact sequence
\begin{equation*}
	\begin{tikzcd}
		0\arrow[r,""]&C_{m_1}\rtimes\Gamma\arrow[r,""]&C_{m_2}\rtimes\Gamma \arrow[r,"\tilde\sigma_{\boundary}"] &C_{m_3}\rtimes\Gamma\arrow[r,""]&0	
	\end{tikzcd}
\end{equation*}
Hence, there is a six term exact sequence
\begin{equation}\label{eq:six-term-k-ell}
	\begin{tikzcd}
		K_0(C_{m_1}\rtimes\Gamma)\arrow[r]&\Ell(\mathcal A_0,\nlBM) \arrow[r,""] &K_0(C_{m_3}\rtimes\Gamma)\arrow[d]\\
		K_1(C_{m_3}\rtimes\Gamma)\arrow[u,""]&K_1(C_{m_2}\rtimes\Gamma)\arrow[l,""] &K_1(C_{m_1})\rtimes\Gamma)\arrow[l].
	\end{tikzcd}
\end{equation}
First of all, \(C_{m_1}\) is \(\Gamma\)-equivariantly isomorphic to \(C_0(T^*Y^\circ)\). Hence, we get that \(K_i(C_{m_1}\rtimes\Gamma)\cong K_i(C_0(T^*Y^\circ)\rtimes\Gamma)\) for \(i=0,1\).

Next, we study \(K_i(C_{m_3}\rtimes\Gamma)\) by comparing it to another cone. Recall the map \(b\colon C(\partial Y)\to\Sigma_{\boundary}\) from \eqref{eq:m_0,m,b} induced by the embedding in the left corner.
Then there is the following commuting \(\Gamma\)-equivariant diagram
\begin{equation*}
	\begin{tikzcd}
		0\arrow[r,""]&\Sigma_{\boundary}\arrow[r,"\id"]&\Sigma_{\boundary}\arrow[r,""] &0\arrow[r]&0\\
		0\arrow[r,""]&C(\partial Y)\arrow[r,"{f\mapsto(f,0)}"]\arrow[u,"b"]&C(\partial Y)\oplus C(\partial Y)\arrow[r,"{(f,g)\mapsto g}"]\arrow[u,"m_3"] &C(\partial Y)\arrow[r,""]\arrow[u,""]&0.
	\end{tikzcd}
\end{equation*}
The partial crossed product sequence of the corresponding cones is
\begin{equation}\label{eq:boundary-cones}
	\begin{tikzcd}
		0\arrow[r,""]&C_{b}\rtimes\Gamma\arrow[r,""]&C_{m_3}\rtimes \Gamma\arrow[r,""] &C(\partial Y)\rtimes\Gamma\arrow[r,""]&0.	
	\end{tikzcd}
\end{equation}
Hence, we see from Lemma  \ref{res:boundary-cone-trivial} that \(K_i(C_{m_3}\rtimes\Gamma)\cong K_i(C(\partial Y)\rtimes\Gamma)\) for \(i=0,1\). Consequently, the six term exact sequence \eqref{eq:six-term-k-ell} becomes
\begin{equation}\label{eq:six-term-k-ell2}
	\begin{tikzcd}
		K_0(C_0(T^*Y^\circ)\rtimes\Gamma)\arrow[r]&\Ell(\mathcal A_0,\nlBM) \arrow[r,""] &K_0(C(\partial Y)\rtimes\Gamma)\arrow[d]\\
		K_1(C(\partial Y)\rtimes\Gamma)\arrow[u,""]&K_1(C_{m_2}\rtimes\Gamma)\arrow[l,""] &K_1(C_0(T^*Y^\circ)\rtimes\Gamma)\arrow[l].
	\end{tikzcd}
\end{equation}
\subsection{Construction of a split}
In the following, we show that \eqref{eq:six-term-k-ell2} splits. First, we describe the homomorphism
\begin{equation*}
	\Ell(\mathcal A_0,\nlBM)\cong K_0(C_{\tilde m_2})\to K_0(C_{\tilde m_3})\to K_0\left(C_{C(\partial Y )\times\Gamma\to 0}\right)\cong K_0(C(\partial Y)\rtimes\Gamma).
\end{equation*}
Using the picture of relative \(K\)-groups, the map \(K_0(C_{\tilde m_2})\to K_0\left(C_{C(\partial Y)\times\Gamma\to 0}\right)\)
is given by 
\begin{equation*}
	[P_Y\oplus P_{\partial Y},Q_Y\oplus Q_{\partial Y},\sigma(D)]\to[P_{\partial Y},Q_{\partial Y}, 0],
\end{equation*}
see, for example, the discussion before Theorem~2.2 in \cite{Has24}.
It is followed by the isomorphism \(K_0(C_{ C(\partial Y)\rtimes\Gamma\to 0})\cong K_0(C(\partial Y)\rtimes \Gamma)\) which is given by \([P_{\partial Y},Q_{\partial Y}, 0]\mapsto [P_{\partial Y}]-[Q_{\partial Y}]\), see \cite{Has24}*{Theorem~2.4(i)}

A homomorphism \(s\colon K_0(C(\partial Y)\rtimes \Gamma)\to \Ell(\mathcal A_0,\nlBM)\) which splits the previous map must therefore associate to a projection \(P_{\partial Y}\) over \(C(\partial Y)\rtimes\Gamma\) an elliptic operator of the form \((D,P_Y\oplus P_{\partial Y},Q_Y\oplus 0)\). This is achieved by using a special boundary value problem associated to a projection over \(C(\partial Y)\rtimes\Gamma\) similarly to \cite{BM71}*{(5.10)}, see also \cite{BS21}*{(77)}.
\begin{lemma}\label{res:special-bm-elemnts}
	There are operators \(A,B\in\overline{\Psi(Y,\partial Y)}\) 
	with the following properties
	\begin{enumerate}
		\item\label{item:1-outside-neighbourhood} \(\sigma_{\interior}(B)=(\sigma_{\interior}(A))^{-1}\) and \(\sigma_{\interior}(A)\) is constant \(1\) outside a neighbourhood of the boundary,
		\item\label{item:Gamma-invar} the principal interior and principal boundary symbol of \(A\) and \(B\) are \(\Gamma\)-invariant,
		\item\label{item:princ-bound} the principal boundary symbols satisfy 
		\begin{equation}\label{eq:special-boundary-symbol}
			\sigma_{\boundary}(A)\sigma_{\boundary}(B)=\begin{pmatrix}
				1 & 0\\ 0 & 1
			\end{pmatrix}\quad\text{ and }\quad 
			\sigma_{\boundary}(B)\sigma_{\boundary}(A)=\begin{pmatrix}
				1 & 0\\ 0 & 0
			\end{pmatrix}.
		\end{equation}
	\end{enumerate}	
\end{lemma}
\begin{proof}
We follow the construction in \cite{BM71}*{Example on p.~35f}. Let \(U\) be a closed neighbourhood of the boundary in \(Y\) which is isometric to \(\partial Y\times[0,1]\) and denote the corresponding coordinates by \((x',x_n)\). Let \(\varphi\in\Schwartz(\R)\) with \(\supp\mathcal F^{-1}(\varphi)\subseteq(-\infty,0]\) and \(\varphi(0)=1\). For \(k\in\N\) with \(k>2\norm{\varphi'}_\infty\) consider the following symbol
\begin{equation*}
	\tilde a_k(x',x_n,\xi',\xi_n)=\frac{\varphi\left(\frac{\xi_n}{k\langle\xi'\rangle}\right)\langle \xi'\rangle+i\xi_n}{\varphi\left(\frac{\xi_n}{k\langle\xi'\rangle}\right)\langle \xi'\rangle-i\xi_n}.
\end{equation*}
Then by \cite{Gru96}*{Theorem~2.5.2} \(\tilde a_k\) is a pseudodifferential symbol of order zero with the transmission condition. 
We choose a cutoff function \(\chi_1\in C^\infty(\R)\) with \(\chi_1(x_n)=1\) for \(x_n\leq \tfrac{1}{3}\) and \(\chi_1(x_n)=0\) for \(x_n\geq \tfrac{2}{3}\) to set 
\begin{eqnarray*}
a_k(x',x_n,\xi',\xi_n)=(\tilde a_k(x',x_n,\xi',\xi_n))^{\chi_1(x_n)}
\end{eqnarray*}
and extend \(a_k\) by \(1\) to a symbol defined an all of \(Y\). Its principal symbol is given on \(U\) by
\begin{equation*}
	a^0_k(x',x_n,\xi',\xi_n)=\left(\frac{\varphi\left(\frac{\xi_n}{k\norm{\xi'}}\right)\norm{\xi'}+i\xi_n}{\varphi\left(\frac{\xi_n}{k\norm{\xi'}}\right)\norm{\xi'}-i\xi_n}\right)^{\chi_1(x_n)}.
\end{equation*}

Furthermore, consider the trace symbol 
\begin{equation}
 t(x',\xi',\xi_n)=\frac{\sqrt{\langle\xi'\rangle}}{\langle\xi'\rangle-i\xi_n}.
\end{equation} The principal operator-valued symbol of  \(t(x',\xi',D_n)\) in the sense of Section \ref{section:BoutetdeMonvelCalculus} is given by \(t^0(x',\xi',D_n)\), where 
\begin{equation*}
t^0(x',\xi',\xi_n)=\frac{\sqrt{\norm{\xi'}}}{\norm{\xi'}-i\xi_n}.
\end{equation*}
Hence, we see that
\begin{equation*}
	s_k=\left(a^0_k,(x',\xi')\mapsto \begin{pmatrix}
		a^0_k(x',0,\xi',D_n)& 0 \\ t^0(x',\xi',D_n))& 0
	\end{pmatrix} \right)\in\Image(\sigma_{\interior},\widehat{\sigma_{\boundary}}).
\end{equation*}
We claim that \((s_k)_{k\in\N}\) converges in \(C(S^*Y)\oplus C(S^*\partial Y,\Bounded(\overline{H_+}\oplus\C))\) to
\begin{equation*}
	s=\left(a^0,(x',\xi')\mapsto \begin{pmatrix}
		a^0(x',0,\xi',D_n)& 0 \\ t^0(x',\xi',D_n)& 0
	\end{pmatrix} \right),
\end{equation*}
where \(a^0=\chi_1(x_n)\tilde a^0(x',x_n,\xi',\xi_n)+(1-\chi_1)(x_n)\) with
\begin{equation*}
	\tilde a^0(x',x_n,\xi',\xi_n)=\frac{\norm{\xi'}+i\xi_n}{\norm{\xi'}-i\xi_n}.
\end{equation*}
Consider first \(C(S^*Y)\) and set
\begin{align*}
	f_k(\xi',\xi_n)&=\varphi\left(\frac{\xi_n}{k\norm{\xi'}}\right)\norm{\xi'}+i\xi_n, &f(\xi',\xi_n)&=\norm{\xi'}+i\xi_n,\\ g_k(\xi',\xi_n)&=\varphi\left(\frac{\xi_n}{k\norm{\xi'}}\right)\norm{\xi'}-i\xi_n, &g(\xi',\xi_n)&=\norm{\xi'}-i\xi_n.
\end{align*}
Then \(f_k\) converges uniformly to \(f\) on \(S^{n-1}\) as for all \(\xi\in S^{n-1}\)
\begin{align*}
	\abs{f_k(\xi)-f(\xi)}=\norm{\xi'}\left(\varphi\left(\frac{\xi_n}{k\norm{\xi'}}\right)-\varphi(0)\right)\leq \norm{\varphi'}_\infty\frac{\abs{\xi_n}}{k}\leq \norm{\varphi'}_\infty\frac{1}{k}.
\end{align*}
Similarly, \(g_k\) converges uniformly to \(g\) on \(S^{n-1}\). As \(\varphi\) is continuous, there is \(k_0\in\N\) such that \(\abs{\varphi(t)}\geq\tfrac{1}{\sqrt 3}\) for all \(\abs{t}\leq\tfrac1{\sqrt{3}k_0}\). Suppose that \(\xi\in S^{n-1}\). If \(\abs{\xi_n}\geq \tfrac{1}{2}\), clearly \(\abs{g_k(\xi',\xi_n)}\geq \tfrac{1}{2}\) holds. If \(\abs{\xi_n}\leq \tfrac{1}{2}\), one has \(\norm{\xi'}^2\geq\tfrac34\), so that for \(k\geq k_0\) also
\begin{align*}
	\abs{g_k(\xi)}^2\geq\varphi^2\left(\frac{\xi_n}{k\norm{\xi'}}\right)\norm{\xi'}^2\geq \frac 14.
\end{align*}
It follows that \(\tfrac{f_k}{g_k}\to \tfrac{f}{g}\) uniformly on \(S^{n-1}\).

Now, we consider \(C(S^*\partial Y,\Bounded(\overline{H_+}\oplus\C))\). Note that for all \((x',\xi')\in S^*\partial Y\), the bounded operator on \(\overline{H_+}\) corresponding to \(a_k^0\) and \(a^0\) are given by the symbols
\begin{equation*}
	\psi_k(\xi_n)= \frac{\varphi\left(\frac{\xi_n}{k}\right)+i\xi_n}{\varphi\left(\frac{\xi_n}{k}\right)-i\xi_n} \quad\text{ and }\quad \psi(\xi_n)= \frac{1+i\xi_n}{1-i\xi_n}.
\end{equation*}
As they act as multipliers, it suffices to show that \(\psi_k\) converges uniformly to \(\psi\) on \(\R\). For the following estimate note that for \(z,w\in\C\)
\begin{equation*}
	\left|\frac{z}{\overline z}-\frac{w}{\overline w}\right|=\frac{2\abs{\Im(z\overline w)}}{\abs{z}\abs{w}}.
\end{equation*}
We  use a similar argument as above to see that there is a \(k_0\in\N\) such that for all \(k\geq k_0\) and all \(\xi_n\in\R\) one has \(\varphi^2(\frac{\xi_n} k)+\xi_n^2\geq \frac{1}{4}\) and obtain 
\begin{align*}
	\abs{\psi_k(\xi_n)-\psi(\xi_n)}=\frac{2\abs{\xi_n}\abs{\varphi\left(\frac{\xi_n}{k}\right)-\varphi(0)}}{\sqrt{\varphi^2\left(\frac{\xi_n}{k}\right)+\xi_n^2}\cdot \sqrt{1+\xi_n^2}}\leq \frac{\norm{\varphi'}_\infty}{k}.
\end{align*}
 As the image of a \(^*\)-homomorphism is closed, \(s\) is contained in the image of \((\sigma_{\interior},\widehat{\sigma_{\boundary}})\). Furthermore, \((\sigma_{\interior},\widehat{\sigma_{\boundary}})\colon \overline{\Psi(Y,\partial Y)}/\Compact \to C(S^*Y)\oplus C(S^*\partial Y,\Bounded(\overline{H_+}\oplus\C))\) is an injective \(^*\)-homomorphism, so there exists \(A\in\overline{\Psi(Y,\partial Y)}\) with \((\sigma_{\interior},\widehat \sigma_{\boundary})(A)=s\).

Similarly as above, one can show the existence of \(B\in\overline{\Psi(Y,\partial Y)}\) 
with 
\begin{equation*}
	(\sigma_{\interior},\widehat {\sigma_{\boundary}})(B)=\left((a^0)^{-1}|_{S^*Y},(x',\xi')\mapsto \begin{pmatrix}
		(a^0)^{-1}(x',0,\xi',D_n)& k^0(x',\xi',D_n) \\ 0& 0
	\end{pmatrix} \right),
\end{equation*}
where \(k^0\) is  given by
\begin{equation*}
	k^0(x',\xi',\xi_n)=\frac{\sqrt{\norm{\xi'}}}{2(\norm{\xi'}+i\xi_n)}.
\end{equation*}
Then it is clear that \ref{item:1-outside-neighbourhood} and \ref{item:Gamma-invar} are satisfied as the metric is \(\Gamma\)-invariant. To check the last property \ref{item:princ-bound} we use the following orthogonal basis \(\{\varphi_l\}_{l\in\N_0}\) of \(\overline{H_+} \)
given by
\begin{equation*}
	\varphi_l(\xi_n)=\frac{(1-i\xi_n)^l}{(1+i\xi_n)^{l+1}}.
\end{equation*}
	Then the operators appearing in the Wiener-Hopf operators for \((x',\xi')\in S^*\partial Y\) act as follows on the basis
\begin{align*}
	a^0(x',0,\xi',D_n)\varphi_l&=\begin{cases}
		\varphi_{l-1} &\text{for }l\geq 1,\\
		0 &\text{for }l=0.
	\end{cases} &(a^0)^{-1}(x',0,\xi',D_n)\varphi_l&=\varphi_{l+1}\\
	t^0(x',\xi',D_n)\varphi_l&= \int \varphi_l\overline{\varphi_0}d\xi_n=\begin{cases}
		0 &\text{for }l\geq 1,\\
		2 &\text{for }l=0.
	\end{cases}
	&k^0(x',\xi',D_n)(1)&=\tfrac12\varphi_0.
\end{align*}
Then it is easily checked that \eqref{eq:special-boundary-symbol} holds.
\end{proof}
Suppose now that \(P\in\Mat_n(C(\partial Y)\rtimes\Gamma)\) is a projection. We first lift \(P\) to a neighbourhood of the boundary. Consider as above the closed neighbourhood \(U\) of the boundary in which \(Y\) is isometric to \(\partial Y\times[0,1]\). Let \(\tilde P\in\Mat_n(C(\partial Y)\rtimes\Gamma\otimes C([0,1]))\) denote the projection which is constant \(P\) for every \(t\in[0,1]\). Since the \(\Gamma\)-action in the normal direction is trivial, one can identify \(C(\partial Y)\rtimes\Gamma\otimes C([0,1])\cong C(\partial Y\times[0,1])\rtimes\Gamma\) by \cref{res:partial-tensor-prod}.
Now choose a cutoff function \(\chi_2\in C^\infty([0,1])\) with \(\chi_2(x_n)=1\) for \(t\leq\tfrac{7}{9}\) and \(\chi_2(x_n)=0\) for \(t\geq\tfrac{8}{9}\). Then \(\chi_2\cdot \tilde P\) can be viewed as an element of \(\Mat_n(C_0(\partial Y\times[0,1)\rtimes\Gamma)\subseteq\Mat_n(C(Y)\rtimes\Gamma)\). Denote by \(\Phi\) the homomorphism
\begin{equation*}\Mat_n(C(Y)\rtimes\Gamma)\oplus \Mat_n(C(\partial Y)\rtimes\Gamma)\cong \Mat_n(C(Y)\rtimes\Gamma\oplus C(\partial Y)\rtimes\Gamma)\hookrightarrow \Mat_n(\overline{\Psi(Y,\partial Y)}\rtimes\Gamma).\end{equation*}
Consequently, we can define using the operators \(A,B\) from \cref{res:special-bm-elemnts}
\begin{align}\label{eq:A_P}
	A_P&= {\Phi(\chi_2\tilde P,P)\cdot} A
	+\Phi(\id_n-\chi_2\tilde P,0)\in\Mat_n(\overline{\Psi(Y,\partial Y)}\rtimes\Gamma),\\
	B_P&= \Phi(\chi_2\tilde P,P) \cdot B
	+\Phi(\id_n-\chi_2\tilde P,0)\in\Mat_n(\overline{\Psi(Y,\partial Y)}\rtimes\Gamma).\label{eq:B_P}
\end{align}
\begin{proposition}
	Suppose that \cref{ass:topologically-free} holds and the partial action of \(\Gamma\) on \(Y\) is topologically free. The map \(P\mapsto (B_P,\Phi(\id_n,P),\Phi(\id_n,0_n))\) induces a split \(s\colon K_0(C(\partial Y\rtimes\Gamma))\to\Ell(\mathcal A_0,\nlBM)\).
\end{proposition}
\begin{proof}
	First, we check that \((B_P,\Phi(\id_n,P),\Phi(\id_n,0_n))\) is elliptic. For the interior symbol we compute, using that the interior symbols of \(A,B\) are \(\Gamma\)-invariant,	
	\begin{align*}
		\sigma_{\interior}(A_P)\sigma_{\interior}(B_P)=\,& (\sigma_{\interior}(A)\chi_2\widetilde P+\id_n-\chi_2\widetilde P)(\sigma_{\interior}(B)\chi_2\widetilde P+\id_n-\chi_2\widetilde P)\\
		=\,& 2\chi_2^2\widetilde P +\chi_2(1-\chi_2)\sigma_{\interior}(A)\widetilde P+\chi_2(1-\chi_2)\sigma_{\interior}(B)\widetilde P +\id_n-2\chi_2\widetilde P\\
		=\,& \id_n+2(\chi_2^2-\chi_2+\chi_2(1-\chi_2))\widetilde P=\id_n.
	\end{align*}
	Here, we used that \(\chi_2(1-\chi_2)\sigma_{\interior}(A)=\chi_2(1-\chi_2)\sigma_{\interior}(B)=\chi_2(1-\chi_2)\) as \(\chi_2\) is constant \(1\) for \(x_n\in[\tfrac{7}{9},\tfrac{8}{9}]\). Similarly, we also get that \(\sigma_{\interior}(B_P)\sigma_{\interior}(A_P)=\id_n\).
	
	Denote by \(\Phi_{\boundary}\) the homomorphism 
	\begin{equation*}\Mat_n(C(\partial Y)\rtimes\Gamma)\oplus \Mat_n(C(\partial Y)\rtimes\Gamma)\cong \Mat_n(C(\partial Y)\rtimes\Gamma\oplus C(\partial Y)\rtimes\Gamma)\hookrightarrow \Mat_n(\Sigma_{\boundary}\rtimes\Gamma)\end{equation*}
	and note that \(\sigma_{\boundary}\circ\Phi=\Phi_{\boundary}\circ r\), where \(r\) is induced by the restriction \((f,g)\mapsto(f|_{\partial Y},g)\).
	For the boundary symbol we employ \eqref{eq:special-boundary-symbol} and the \(\Gamma\)-invariance of the boundary symbols to compute
	\begin{align*}
		\sigma_{\boundary}(A_P)\sigma_{\boundary}(B_P)&= (\sigma_{\boundary}(A)\Phi_{\boundary}(P,P)+\Phi_{\boundary}(\id_n-P,0_n))(\sigma_{\boundary}(B)\Phi_{\boundary}(P,P)+\Phi_{\boundary}(\id_n-P,0_n))\\
		&= \sigma_{\boundary}(A)\sigma_{\boundary}(B)\Phi_{\boundary}(P,P)+\Phi_{\boundary}(\id_n-P,0) =\Phi_{\boundary}(\id_n,P),\\
		\sigma_{\boundary}(B_P)\sigma_{\boundary}(A_P)&= \sigma_{\boundary}(B)\sigma_{\boundary}(A)\Phi_{\boundary}(P,P)+\Phi_{\boundary}(\id_n-P,0_n)=\Phi_{\boundary}(P,0_n)+\Phi_{\boundary}(\id_n-P,0_n)\\
		&=\Phi_{\boundary}(\id_n,0).
	\end{align*}	
	Moreover, the map is additive, that is for \(P\in\Mat_n(C(\partial Y)\rtimes \Gamma)\) and \(Q\in\Mat_m(C(\partial Y)\rtimes \Gamma)\) one has 
	\begin{equation*}
		(B_P,\Phi(\id_n,P),\Phi(\id_n,0))+(B_Q,\Phi(\id_m,Q),\Phi(\id_m,0))=(B_{P\oplus Q},\Phi(\id_{n+m},P\oplus Q),\Phi(\id_{n+m},0)).
	\end{equation*} 
	When \(P=0_n\), then \((B_0,\Phi(\id_n,0),\Phi(\id_n,0))=(\Phi(\id_n,0),\Phi(\id_n,0),\Phi(\id_n,0))\) is a trivial elliptic triple. Suppose now that there is a homotopy \(P_t\) for \(t\in[0,1]\) of projections in \(\Mat_n(C(\partial Y)\rtimes\Gamma)\). Then the associated elliptic triples
	\((B_{P_t},\Phi(\id_n,P_t), \Phi(\id_n,0_n))\) define a homotopy as well. 
	Hence, we see that the assignment gives a well-defined map \(K_0(C(\partial Y)\rtimes\Gamma)\to\Ell(\mathcal A_0,\nlBM)\).
	
	Lastly, it is clear from the description of the map \(\Ell(\mathcal A_0,\nlBM)\to K_0(C(\partial Y)\rtimes\Gamma)\) above, that the described map is indeed a split.
 \end{proof}
 \begin{theorem}
 	Suppose that \cref{ass:topologically-free} holds and the partial action of \(\Gamma\) on \(Y\) is topologically free. The map \(K_0(C_0(T^*Y^\circ)\rtimes\Gamma)\to \Ell(\mathcal A_0,\nlBM)\) and the split \(s\colon K_0(C(\partial Y)\rtimes\Gamma)\to \Ell(\mathcal A_0,\nlBM)\) yield an isomorphism
 	\begin{equation}\label{eq:sum-decomp-of-ell}
 		K_0(C_0(T^*Y^\circ)\rtimes\Gamma)\oplus K_0(C(\partial Y)\rtimes\Gamma)\cong\Ell(\mathcal A_0,\nlBM).
 	\end{equation}
 \end{theorem}
 \begin{proof}
 	As the homomorphism \(\Ell(\mathcal A_0,\nlBM)\to K_0(C(\partial Y)\rtimes\Gamma)\) splits, it suffices to show that the map \(K_0(C_0(T^*Y^\circ)\rtimes\Gamma)\to \Ell(\mathcal A_0,\nlBM)\) is injective. Consider the following commuting diagram in \(K\)-theory induced by \eqref{eq:ses-mapping-cones}
 	\begin{equation*}
 		\begin{tikzcd}
 			K_1(\ker(\sigma_{\boundary})/\Compact\rtimes\Gamma) \arrow[d]\arrow[r]&K_1(\Sigma\rtimes\Gamma)\arrow[d]\\
 			K_0(C_{m_1}\rtimes\Gamma)\arrow[r] & K_0(C_{m_2}\rtimes\Gamma).
 		\end{tikzcd}
 	\end{equation*}
 	Under the previous identifications, the bottom map is the map in question \(K_0(C_0(T^*Y^\circ)\rtimes\Gamma)\to \Ell(\mathcal A_0,\nlBM)\). By the observations in \cref{res:boundary-cone-trivial}, \cref{res:split-1} and \cref{res:split-2}, the above diagram can be identified with 
 	\begin{equation}\label{eq:diagram-injective}
 		\begin{tikzcd}
 			K_1(C(Y^\circ)\rtimes\Gamma)\oplus K_0(C_{m_1}\rtimes\Gamma) \arrow[d,"\pr_2"]\arrow[r,"\rho\oplus\phi"]&K_1(C(Y)\rtimes\Gamma)\oplus K_0(C_m\rtimes\Gamma)\arrow[d,"\psi\circ\pr_2"]\\
 			K_0(C_{m_1}\rtimes\Gamma)\arrow[r] & K_0(C_{m_2}\rtimes\Gamma).
 		\end{tikzcd}
 	\end{equation}
 	Here, \(\rho\) is induced by the inclusion \(C_0(Y^\circ)\to C(Y)\) and \(\phi\colon K_0(C_{m_1}\rtimes\Gamma)\to K_0(C_m\rtimes\Gamma)\) is the isomorphism from \cref{res:boundary-cone-trivial}. Furthermore, \(\psi\colon K_0(C_m\rtimes\Gamma)\to K_0(C_{m_2}\rtimes\Gamma)\) denotes the map induced by the mapping cone sequence of
 	\begin{equation*}
 		\begin{tikzcd}
 			0\arrow[r,""]&\Sigma\arrow[r,"\id"]&\Sigma\arrow[r,""] &0\arrow[r]&0\\
 			0\arrow[r,""]&C(Y)\arrow[r,"{f\mapsto(f,0)}"]\arrow[u,"m"]&C(Y)\oplus C(\partial Y)\arrow[r,"{(f,g)\mapsto g}"]\arrow[u,"m_2"] &C(\partial Y)\arrow[r,""]\arrow[u,""]&0.
 		\end{tikzcd}
 	\end{equation*}
 	As the lower sequence splits, we see from the following diagram and the four-lemma that \(\psi\) is injective
 		\begin{equation*}
 		\begin{tikzcd}
 			K_1(\Sigma\rtimes\Gamma)\arrow[r]\arrow[d,equal]&K_0(C(Y)\rtimes\Gamma)\arrow[r]\arrow[d] & K_0(C_m\rtimes\Gamma)\arrow[r]\arrow[d,"\psi"]&K_0(\Sigma\rtimes\Gamma)\arrow[d,equal]\\
 			K_1(\Sigma\rtimes\Gamma)\arrow[r] & K_0(C(Y)\rtimes\Gamma\oplus C(\partial Y)\rtimes\Gamma)\arrow[r]& K_0(C_{m_2}\rtimes\Gamma)\arrow[r]&K_0(\Sigma\rtimes\Gamma).
 		\end{tikzcd}
 	\end{equation*}
 	Then by the commutativity of \eqref{eq:diagram-injective} the bottom map must be injective as well. 
 \end{proof}
 \begin{remark}
 	In \cite{TX06} a Chern character for an \'etale proper groupoid \(G\) is studied and it induces an isomorphism \(K_i(C^*( G))\otimes\C\cong H^*_c( G)\). 
 	
 	In our case, recall that for a partial action of \(\Gamma\) on a space \(X\) one has \(C_0(X)\rtimes\Gamma\cong C^*(X\rtimes\Gamma)\), see \cref{sec:partial-crossed-products}. Hence, the result of Tu and Xu can be used to compute the summands in \eqref{eq:sum-decomp-of-ell} if the partial action groupoids \(T^*Y^\circ\rtimes \Gamma\) and \(\partial Y\rtimes \Gamma\) are proper. They are always \'etale under our assumptions, see \cref{rem:etale}.

 	\comment{The partial action on the cutted space \(Y\) (and similarly \(T^*Y^\circ\) and \(\partial Y\)) should always be proper. Namely, \(Y\) is a disjoint union, denote its connected components by \(Y_i\) with \(i\in I\). Then it suffices to show for \(K_1,K_2\subseteq Y\) compact that \((r\times s)^{-1}(K_1\times K_2)\subseteq Y\times\Gamma\) is compact. 
 	Note that \(K_j\) for \(j=1,2\) must be the disjoint union of \(K_{j,i}\subseteq Y_i\) compact for \(i\in I\). Hence, \((r\times s)^{-1}(K_1\times K_2)\) is a finite union of sets of the form
 	\(\theta_{g}(Y_{1,k})\cap Y_{2,i}\times\{g^{-1}\}\), where \(g\in\Gamma\) is such that \(\theta_g(Y_k)=Y_i\) and, thus, compact.
 	}
 \end{remark}

  Next, we will show that the summand \(K_0(C(\partial Y)\rtimes\Gamma)\) of \(\Ell(\mathcal{A}_0,\mathcal A)\) does not contribute to the Fredholm index as in the case without partial actions. Recall that we assume here that $\Gamma$ is finitely generated and of polynomial growth. 
 
 In the proof of \cref{thm:index_sp_zero}, below, it will be convenient to replace the projection $P\in \Mat_n(C(\partial Y)\rtimes \Gamma)$ by a `smooth' version that can be taken arbitrarily close to $P$. We write $\Mat_n(C^\infty(\partial Y)\rtimes\Gamma)$ for the smooth crossed product defined in \cref{def:smooth-crossed-product}.
 
 \begin{lemma}\label{lem:replaceP}
Suppose that $\Gamma$ is finitely generated and of polynomial growth. Let $P\in \Mat_n(C(\partial Y) \rtimes \Gamma)$ be a projection. Then we can find a projection $\tilde P\in \Mat_n(C^\infty(\partial Y)\rtimes \Gamma)$ arbitrarily close to $P$. 
 \end{lemma}
 
 \begin{proof}
 Choose $P'\in \Mat_n(C^\infty(\partial Y)\rtimes \Gamma)$ close to $P$. 
 By possibly replacing $P'$ by $(P'+P^{\prime*})/2$ we may assume $P'$ to be symmetric. In general, $P'$ will not be a projection; however, its spectrum will be close to $0$ and $1$. We can therefore find a contour $\mathscr C$ that simply encloses the eigenvalues  near $1$ but none of the eigenvalues near zero. Then 
 $$\tilde P = \frac{1}{2\pi i} \int_{\mathscr C} (P'-\lambda)^{-1} \, d\lambda$$
 will be a projection close to $P$. 
 
 Now $\Mat_n(C^\infty(\partial Y)\rtimes \Gamma)$ is spectrally invariant in $\Mat_n(C(\partial Y)\rtimes \Gamma)$ (the argument will follow, below). Hence  $(P'-\lambda)^{-1} \in \Mat_n(C^\infty(\partial Y)\rtimes \Gamma)$ and, since inversion is continuous in $\Mat_n(C^\infty(\partial Y)\rtimes \Gamma)$ by \cite{Wae71},   the integral actually converges in $\Mat_n(C^\infty(\partial Y)\rtimes \Gamma)$, so that 
 $\tilde P$ is a projection in $\Mat_n(C^\infty(\partial Y)\rtimes \Gamma)$.
 
 In order to see the spectral invariance, we first recall that the partial action by $\Gamma$ on $\partial Y$ is globalized by the \(\Gamma\)-action on \(\partial V\). According to  \cite{Sch93}*{Corollary 7.16} \comment{or cite Anton's book Prop. 1.33 for noncompact case?}, the smooth crossed product $C^\infty(\partial V)\rtimes \Gamma$ is spectrally invariant in $C(\partial V)\rtimes \Gamma$, since $\Gamma$ is finitely generated and of polnomial growth. Note that $C^\infty(\partial V)\rtimes \Gamma$ is denoted by $L^\tau_1(\Gamma,C^\infty(\partial V))$ in \cite{Sch93}*{Definition 6.1}.
We then apply \cite{Sch93}*{Theorem 3.2(2)} with 
\begin{itemize}
\item $A_1=C^\infty(\partial Y)\rtimes \Gamma$ the smooth partial crossed product,
\item $A= C^\infty(\partial V)\rtimes \Gamma$ the smooth globalized crossed product,
\item $B_1=C(\partial Y)\rtimes \Gamma$ the partial crossed product,
\item $B= C(\partial V)\rtimes \Gamma$ the globalized crossed product.
\end{itemize} 
 Since $A$ is spectrally invariant in $B$, $A_1$ will be spectrally invariant in $B_1$. This completes the argument. 
 \end{proof}

 \begin{theorem}\label{thm:index_sp_zero}
 	Suppose that \cref{ass:topologically-free} holds, the partial action of \(\Gamma\) on \(Y\) is topologically free and $\Gamma$ is finitely generated and of polynomial growth. For every projection \(P\in\Mat_n(C(\partial Y)\rtimes\Gamma)\), the associated elliptic Boutet de Monvel problem 
 	\begin{equation*}
 		s(P)=(B_P,\Phi(\id_n,0_n),\Phi(\id_n,P)),
 	\end{equation*}
 	where \(B_P\) is defined in \eqref{eq:B_P}, has index zero. 	
 \end{theorem}
 \begin{proof}
 	Equivalently, we can show that the Fredholm inverse \(A_P\colon L^2(Y,\C^n)\to L^2(Y,\C^n)\oplus P L^2(\partial Y,\C^n)\) has index zero. Since $A_P$ is a Fredholm operator we can replace $P$ by a projection in $\Mat_n(C^\infty(\partial Y)\rtimes_{\smooth} \Gamma)$ so close to $P$ that the index does not change. 
	So we may assume from the beginning that $P\in \Mat_n(C^\infty(\partial Y)\rtimes_{\smooth} \Gamma)$. 
	
Our strategy then is to compose $A_P$ from the left and the right with Fredholm operators of index zero in order to obtain a boundary value problem 
$$(\tilde D_P,R_P): H^1(Y,\C^n) \to L^2(Y,\C^n) \oplus PH^{1/2}(\partial Y, \C^n),$$ 
where $R_Pu=P\gamma_0u)$, and such that $(\tilde D_P, R_P)$ is homotopic through elliptic elements to a boundary value problem $(D_P,R_P)$ for which one knows that the index is zero. For details see \cref{rem:replaceP}, below.

We fix first order pseudodifferential operators \(\Lambda_Y\) on \(Y\) with principal symbol $\norm{\xi}$ and \(\Lambda_{\partial Y}\) on \(\partial Y\) with principal symbol $\norm{\xi'}$ such that $\Lambda_Y$ and $\Lambda_{\partial Y}$ are $\Gamma$-invariant. Then we define the operator
 		$$
 	\Lambda_{-}=\Lambda^{\interior}+\Lambda^{\boundary}: C^\infty(  Y,\mathbb{C}^n)\to C^\infty (Y,\mathbb{C}^n),
 	$$
 	where
 	$$
 	\Lambda^{\interior}=(1-\chi_3)(\Lambda_Y\otimes 1_n) (1-\chi_3) \
 	$$
and 
 	$$
 	\Lambda^{\boundary}=\chi_3\left[(-iD_n+\Lambda_{\partial Y})\otimes 1_n\right].
 	$$
Here, \(\chi_3\in C^\infty_c([0,1))\) is chosen  in such a way that \(\chi_2\chi_3 =\chi_2\), for the function  \(\chi_2\) in \cref{res:special-bm-elemnts}.
By \cite{Gru96}*{Theorem~2.5.2},  \(\Lambda_{-}\) defines an invertible bounded operator \(H^1(Y,\C^n)\to L^2(Y,\C^n)\).
We next consider the composition 
\begin{eqnarray*}
(\id\oplus {P}\Lambda_{\partial Y}^{-1/2})\circ A_P\circ \Lambda_{-}\colon H^1(Y,\C^n)\to L^2(Y,\C^n)\oplus P H^{1/2}(\partial Y, \C^n).
\end{eqnarray*}
Clearly, $\sigma_{\interior}(A_P\circ\Lambda_-)$ is invertible on $S^*Y$.
Using the fact that $\chi_2 \chi_3=\chi_2$ and the function $\chi_1$ from \cref{res:special-bm-elemnts}, we find 
\begin{eqnarray*}
\sigma_{\interior}(A_P\circ\Lambda_-) 
&=&
\chi_2P(i\xi_n+\|\xi'\|) ^{\chi_1}(-i\xi_n+\|\xi'\|)^{1-\chi_1}\otimes 1_n\\
&&+\chi_2(1-P) (-i\xi_n+\|\xi'\|)\otimes 1_n\\
&&+(\chi_3-\chi_2) (-i\xi_n+\|\xi'\|) \otimes 1_n\\
&& +(1-\chi_3)\|\xi\|(1-\chi_3) \otimes 1_n. 
\end{eqnarray*}

Next consider the boundary symbol of $A_P\circ \Lambda_-$. Recall that the  boundary symbol of $T$ is $t_0(x',\xi',\xi_n) =  \|\xi'\|^{1/2}(\|\xi'\|-i\xi_n)^{-1}$. Moreover,  \(R_P(u)=Pu|_{\partial Y}= P\gamma_0 u\) with the boundary evaluation operator $\gamma_0$, where 
the boundary symbol of $\gamma_0$ is the constant function $1$.
We see that 
\begin{eqnarray*}
\lefteqn{\sigma_{\boundary}(P(\Lambda_{\partial Y}^{-1/2}\otimes 1_n) (T(\Lambda_{\partial Y}-iD_n) \otimes 1_n)}\\
&=& \sigma_{\boundary}(P(\Lambda_{\partial Y}^{-1/2} T (\Lambda_{\partial Y}-iD_n))\otimes 1_n)  \\
&=&P(\|\xi'\|^{-1/2} t_0(x',\xi',\xi_n) (\|\xi'\|-i\xi_n)\otimes 1_n)\\
&=&\sigma_{\boundary}(P\gamma_0) .
\end{eqnarray*}

In \cite{Hoe83c}*{Proposition~20.3.1}, H\"ormander has shown that the boundary value problem 
\begin{eqnarray*}
(D_P,R_P): H^1(Y) \oplus L^2(Y)\oplus H^{1/2} (\partial Y) 
\end{eqnarray*}
with 
\begin{eqnarray*}
D_P &=& \chi_3[(iD_n+\Lambda_{\partial Y})P + (-iD_n+\Lambda_{\partial Y})(1-P)]+ (1-\chi_3)\Lambda (1-\chi_3)
\end{eqnarray*}
and $R_P$ as above, has index zero. 
We notice that the interior symbols of the pseudodifferential part $\tilde D_P$ of $A_P\circ \Lambda_-$ and of $D_P$ coincide on the sets where $\chi_1=1$ or $1-\chi_3=1$. 
We can then find a homotopy through invertible symbols joining $\sigma_{\interior}(A_P\circ  \Lambda_-)$ and $\sigma_{\interior}(D_P)$.  

Finally we observe that $\Lambda_-: L^2(Y,\C^n) \to H^1(Y,\C^n)$ is an isomorphism and that 
\begin{eqnarray*}
P\Lambda_{\partial Y}^{-1/2}: PL^2(\partial Y,\C^n) \to PH^{1/2}(\partial Y,\C^n	) 
\end{eqnarray*}
is a Fredholm operator of index zero, see \cref{PLambda}, below. 
Combining these facts we conclude that  the index of $A_P$ and hence also that of $ s(P)$ is zero.  
 \end{proof}

 \begin{lemma} \label{PLambda}
 Let $\Lambda_{\partial Y}$ be a positive  first order pseudodifferential operator on $\partial Y$ such that $\Lambda_{\partial  Y}^{-1/2}: L^2(\partial Y)\to H^{1/2}(\partial Y)$ is an isomorphism. E.g., choose $\Lambda_{\partial Y}=(1-\Delta)^{1/2}$ where $\Delta$ is the Laplacian with respect to the $\Gamma$-invariant metric. Moreover, let $P\in \Mat_n(C^\infty(\partial Y)\rtimes\Gamma)$.
 Then the map  $P\Lambda_{\partial Y}^{-1/2}: PL^2(\partial Y) \to PH^{1/2} (\partial Y) $ is a Fredholm operator of index zero. 
 \end{lemma} 	
 
 \begin{proof} 
 Write 
 $$\Lambda_{\partial  Y}^{-1/2} = P\Lambda_{\partial  Y}^{-1/2}P+P\Lambda_{\partial  Y}^{-1/2}(1-P) + (1-P)\Lambda_{\partial  Y}^{-1/2}P + (1-P)\Lambda_{\partial  Y}^{-1/2}(1-P)$$
 and 
 $P\Lambda_{\partial  Y}^{-1/2}(1-P) = P [P,\Lambda_{\partial  Y}^{-1/2}](1-P)$.
We note that \(P\in \Psi^0_\Gamma(\partial Y)\) and \(\Lambda^{-1/2}_{\partial Y}\in \Psi^{-1/2}_\Gamma(\partial Y)\), and that their principal symbols commute as the principal symbol of \(\Lambda^{-1/2}_{\partial Y}\in \Psi^{-1/2}_\Gamma(\partial Y)\) is \(\Gamma\)-invariant. Hence, the operator \([P,\Lambda_{\partial  Y}^{-1/2}]\) defines a compact operator \(L^2(\partial Y)\to H^{1/2}(\partial Y)\) by \cref{res:smooth-calculus}. The same is true for $(1-P)\Lambda^{-1/2}_{\partial Y} P$. Hence 
 \begin{eqnarray*}
 P\Lambda_{\partial  Y}^{-1/2}P+ (1-P)\Lambda_{\partial  Y}^{-1/2}(1-P) \equiv \Lambda_{\partial  Y}^{-1/2} \text{\;\rm  mod } \Compact
\end{eqnarray*}
is a Fredholm operator of index zero. We deduce that both $P\Lambda^{-1/2}_{\partial Y}P$ and $(1-P)\Lambda^{-1/2}_{\partial Y}(1-P)$ are Fredholm operators, and the sum of their indices is zero.
 
Now, $P\Lambda_{\partial Y}^{-1/2}P = (\Lambda^{-1/4}_{\partial Y}P)^*(\Lambda^{-1/4}_{\partial Y}P)$ is a positive operator.
 If $P\Lambda_{\partial Y}^{-1/2} Pu=0$ for some $u\in PL^2(\partial Y)$, then $\Lambda^{-1/4}Pu=0$, and the injectivity of $\Lambda_{\partial Y}^{-1/4}$ implies that  $Pu=0$, hence $u=0$, since $u\in PL^2(\partial Y)$. Thus $P\Lambda^{-1/2} P$ is injective. The same argument furnishes the injectivity $(1-P)\Lambda^{-1/2}_{\partial Y}(1-P)$. As a consequence, both $P\Lambda^{-1/2}_{\partial Y} P$ and  $(1-P)\Lambda^{-1/2}_{\partial Y}(1-P)$ have index zero. 
 \end{proof} 
 
 	
\bibliography{references.bib}
\end{document}